\documentclass[compress, preprint,9pt]{elsarticle}

\usepackage{bbm}
\usepackage{mathrsfs}
\usepackage{graphicx}
\usepackage{amstext}
\usepackage{amsmath,amssymb}
\usepackage{graphicx}
\usepackage{esint,algorithm}
\usepackage{color}

\usepackage[abs]{overpic}
\usepackage{graphicx}
\usepackage{subfigure}
\usepackage{stmaryrd}
\usepackage{amsfonts, color, amssymb}
\usepackage{amsmath}
\usepackage{amsthm}
\usepackage{epsfig}
\usepackage{psfrag}
\usepackage{bm}
\usepackage{paralist}
\usepackage{algorithm}
\usepackage{algorithmicx}
\usepackage{algpseudocode}
\usepackage{appendix}
\usepackage{caption}
\usepackage{hyperref}
\usepackage{color}

\usepackage[misc]{ifsym}
\newcommand{\bx}{{\bf x}}

\newcommand{\bw}{{\bf w}}

\newcommand{\bv}{{\mathbf v}}

\def\bg{{\bf g}}

\def\bn{{\bf n}}
\def\3bar{{|\hspace{-.02in}|\hspace{-.02in}|}}
\def\ljump{{[\![}}
\def\rjump{{]\!]}}

\allowdisplaybreaks[4]
\def\jump#1{{ \left[ \!\!\left[ #1 \right]\!\!\right] }}
\numberwithin{equation}{section}

\usepackage[a4paper, total={6in,8in}]{geometry}

\allowdisplaybreaks[4]
\newtheorem{theorem}{Theorem}[section] %
\newtheorem{definition}{Definition}[section] %

\newtheorem{example}{Example}[section]

\newtheorem{lemma}{Lemma}[section]

\begin{document}
	\title{{\color{black}The Immersed Skeletal Finite Element Method for Elliptic Interface Problems}}

\author[mymainaddress]{Lin Yang}
\ead{yanglin25@amss.ac.cn}
\author[mysecondaddress]{Qilong Zhai\corref{mycorrespondingauthor}}
\cortext[mycorrespondingauthor]{Corresponding author}
\ead{zhaiql@jlu.edu.cn}

\address[mymainaddress]{Academy of Mathematics and Systems Science, Chinese Academy of Sciences, Beijing 100190, P. R. China} 
\address[mysecondaddress]{School of Mathematics, Jilin University, Changchun 130012, Jilin, P. R. China}

	\begin{abstract}
		In this paper, we present a new immersed finite element scheme for solving elliptic interface problems on unfitted meshes by combining the {\color{black}skeletal finite element method (FEM)} with the standard {\color{black}FEM. The skeletal FEM is used for the interface elements. In other words, we take piecewise functions as the unknowns inside the interface element and on its boundary.} We employ the immersed finite element functions as interior functions that precisely satisfy the interface conditions. On the interface edges, we define two boundary functions to capture the discontinuity. The {Lagrange element} is used for the non-interface elements. The proposed scheme is simple and flexible. We prove that this scheme achieves optimal convergence orders in both the $H^1$ norm and $L^2$ norm. {Numerical} experiments are presented to demonstrate the efficiency and accuracy of the proposed method.
	\end{abstract}


	\maketitle

\section{Introduction}
Assume that the domain $\Omega$  is a bounded domain in $\mathbb{R}^2$. And the domain is divided by an interface $\Gamma$ into two subdomains, $\Omega_1$  and $\Omega_2$. That is $\Gamma = \partial \Omega_{1} \cap  \partial \Omega_{2}$.  In this paper, we consider the following second order elliptic interface problems:
\begin{align}
	-\nabla \cdot (\beta \nabla u) &= f, \quad \text{in} \ \Omega_1 \cup \Omega_2, \label{Elliptic_problem}\\
	u &= g , \quad \text{on} \ \partial \Omega, \label{Dirichlet_Boundary}
\end{align}
where $\beta$ is a piecewise constant
\begin{eqnarray*}
	\beta(x,y)=\left\{
	\begin{array}{rcl}
		\beta_1, &  (x,y) \in \Omega_1, \\
		\beta_2, & (x,y) \in \Omega_2.
	\end{array}\right.
\end{eqnarray*}
And we assume $\beta_2 \geqslant \beta_1 > 0$. Set $u_1=u|_{\Omega_{1}}$ and $u_2=u|_{\Omega_{2}}$. The following conditions are satisfied on the interface $\Gamma$:
\begin{align}
	\ljump u \rjump_{\Gamma}&:= u_1 - u_2 =0, \quad\quad\quad\quad\quad\quad\quad\,\,\, \text{on} \ \Gamma, \label{Interface_condition_1}\\
	\ljump \beta \nabla u \cdot \bn \rjump_{\Gamma} &:= \beta_1 \nabla u_1 \cdot \bn - \beta_2 \nabla u_2 \cdot \bn = 0, \quad \text{on}\ \Gamma,\label{Interface_condition_2}
\end{align}
where $\bn$ is the unit outward normal vector on the interface $\Gamma$ pointing from $\Omega_1$ into $\Omega_2$.

For interface problems, numerical methods can be classified into two types based on the relationship between the meshes and the interface location: methods on fitted meshes and unfitted meshes. On fitted meshes, we use numerical methods for {non-interface problems} to solve the interface problems, and it is easy to prove the optimal error estimates \cite{FEM_Interface_1,FEM_Interface_2,FEM_Interface_3,FEM_fitted_2010_2}. However, when solving the moving interface problem on fitted meshes,
the meshes have to be updated to capture the evolving interface, significantly increasing the computational cost. Unfitted meshes are used to avoid this issue. For numerical methods on unfitted meshes, one method is to incorporate interface conditions and extended jump conditions directly into the numerical scheme, such as the immersed interface method \cite{IIM2,li2006immersed} and the CutFEM method \cite{CutFEM1,CutFEM2}, etc. Another method is to modify shape function on the interface elements to satisfy the interface conditions and extended jump conditions, such as the multi-scale finite element method \cite{Multi_scale_FEM1,Multi_scale_FEM2}, the partition of unity method \cite{PIU1,PIU2}, the extended finite element method \cite{ExtendedFEM1,ExtendedFEM2}, generalized finite element method \cite{GFEM1,GFEM2,GFEM3,GFEM4} and immersed finite element method to be used in this article. 

The immersed finite element (IFE) method uses {standard} finite element on non-interface elements and IFE functions on the interface elements. The IFE functions are constructed to satisfy jump conditions. Based on different choices of degrees of freedom, researchers proposed various lowest-order IFE functions in \cite{IFEM_Loworder9,IFEM_Loworder2,IFEM_Loworder11,IFEM_Loworder5,IFEM_Loworder3,IFEM_Loworder1,IFEM_Loworder4,Partially_penalized_IFEM_2015,IFEM_Loworder6,IFEM_Loworder7,IFEM_Loworder10}. When constructing the high order IFE functions, we {can} impose one of the following two sets of extended jump conditions \cite{least_square2017} (for integer $k \geqslant 2$) to ensure the uniqueness of IFE functions:
\begin{itemize}
	\item Laplacian extended jump conditions:
	\begin{eqnarray}
		\jump{	\beta \frac{\partial^{j-2} \Delta u }{\partial \bn^{j-2}}
		}_{\Gamma} =0, \, j=2,3, \cdots, k. \label{Laplacian_extend_interface_conditions}
	\end{eqnarray}
	\item Normal extended jump conditions:
	\begin{eqnarray}
		\jump{	\beta \frac{\partial^{j} u }{\partial \bn^{j}}
		}_{\Gamma} =0, \, j=2,3, \cdots, k. \label{Normal_extend_interface_conditions_2}
	\end{eqnarray}
\end{itemize}
For the curved interface, {\cite{IFEM_Loworder9,IFEM_Loworder11,IFEM_Loworder5,IFEM_Loworder3,Partially_penalized_IFEM_2015,IFEM_Loworder7}} used straight segments to approximate curved interfaces. However, the geometric errors limits the accuracy of the high order IFE functions. Therefore, researchers construct these IFE functions based on actual interface. Based on the above two extended jump conditions on the actual interface,  \cite{least_square2017} constructed the high order IFE functions by using the least squares method. And the interpolation estimates of the IFE function were verified by numerical experiments. In \cite{Cauchy_mapping_2019}, the authors constructed the IFE functions that satisfy the interface conditions and Laplacian extended jump conditions by solving a local Cauchy problem on the interface elements. They proved the existence and uniqueness of IFE functions. And the optimal convergence orders of interpolation estimates and error estimate were obtained.
Then, \cite{Enrich_nonhomogeneous_2023} applied this method to solve the second order elliptic interface problems with non-homogeneous jump condition.  In the above papers, the jump conditions are approximately satisfied. In \cite{Frenet_Serret_2024}, by using the Frenet-Serret frame of the interface, the authors mapped interface to a straight line through a non-affine transformation, adjusting the interface conditions accordingly. Thus, they  constructed IFE functions that precisely satisfy the interface conditions based on the geometric properties of the interface.  The authors also proved that the interpolation estimates of these IFE functions achieve optimal convergence orders. {\cite{IFE_2025} developed the critical trace inequality for the IFE functions and presented optimal error estimates in both $L^2$ and energy norms.}

Another challenge of the IFE method is developing a numerical scheme with optimal convergence orders. Since the IFE functions are independently constructed on each element, discontinuities arise between the elements, directly affecting the accuracy of the numerical scheme. If this discontinuity is not properly addressed, optimal convergence orders may not be achieved. To address this issue, the authors in \cite{Partially_penalized_IFEM_2015} introduced penalty terms on the interface edges to {reduce} the discontinuity of the IFE functions, thus developing a partially penalized IFE method. To avoid adding penalty terms to numerical scheme,  \cite{NCFEM_2018,NCFEM_2013} used the integral values on the interface edges as degrees of freedom, which reduce the discontinuity of IFE functions between elements. However, in \cite{NCFEM_2023}, the authors proved that when the tangential derivative of the exact solution and jump in the coefficient are nonzero on the interface, a numerical scheme without penalty terms cannot achieve optimal convergence order. For this problem, they introduced a lifting operator on the interface edges and developed a parameter-free nonconforming finite element method to ensure optimal convergence order.

In this paper, we employ {\color{black}the skeletal FEM} for the interface elements. {\color{black}The skeletal FEM is the hybridized discontinuous Galerkin FEM with the Lehrenfeld-Sch\"{o}berl (LS) stabilization. The key feature of this method takes the discontinuous piecewise polynomials in the cell and edges as the unknowns. In other words, the approximate function includes two parts: interior functions and boundary functions.} And the boundary functions may not necessarily be related to the trace of interior function. In our work, {referring to the constructed IFE function in \cite{Frenet_Serret_2024}, the interior function for the interface elements is defined.} To address the discontinuity along the element boundaries, we define two boundary functions on the edges that intersect the interface. When the standard finite element method is used to solve the second order elliptic problem, the optimal convergence orders were obtained. And this method is easy to implement. Therefore, the standard finite element method is used for non-interface elements. We propose a new numerical scheme for this interface problem that combines the flexibility of the {\color{black}skeletal} method with the simplicity of standard finite element method, {\color{black} while avoiding the computational challenges associated with excessive degrees of freedom}. And we prove that the error achieves optimal convergence orders in both the $H^1$ norm and $L^2$ norm.

This paper is organized as follows. In Section 2, we focus on the geometric properties of the interface functions. We construct the {\color{black}skeletal} immersed finite element function in Section 3. Section 4 defines the global finite element space and {\color{black}corresponding} differential operators, and present the corresponding numerical scheme. We also analyze the existence and uniqueness of the proposed numerical scheme. In Sections 5, {\color{black}by the projection property of the gradient operator}, the error equation is presented. In Section 6 and Section 7, we demonstrate that the errors achieve optimal convergence orders in the $H^1$ norm and $L^2$ norm. Finally, in Section 8, we provide several numerical examples to validate the effectiveness of our method.

\section{Preliminaries}
For every measurable subset $\Omega$, we introduce the following broken Sobolev space:
\begin{eqnarray*}
	PH^k(\Omega)=\{u: u|_{\Omega_1} \in H^k(\Omega_1)\, \text{and}\,u|_{\Omega_2} \in H^k(\Omega_2)
	\}.
\end{eqnarray*}
The norms and semi-norms used on $PH^k(\Omega)$ are
$$
\|\cdot\|_{k,\Omega}^2=\|\cdot\|_{k,\Omega_1}^2+\|\cdot\|_{k,\Omega_2}^2.
$$

Let $\mathcal{T}_h$ be the shape-regular {triangular meshes} of the domain $\Omega$. And the meshes do not align with the interface. Let $\mathcal{T}_h^n$ denote the set of non-interface elements, and $\mathcal{T}_h^I$ represent the set of interface elements, which intersect with the interface $\Gamma$. That is, $\mathcal{T}_h=\mathcal{T}_h^n \cup \mathcal{T}_h^I$ (see Figure \ref{unfitted_meshes_IEI}). In addition to the shape-regular conditions, the interface elements also need to satisfy the following conditions:
\begin{itemize}
	\item The interface $\Gamma$ can intersect an edge of an interface element at most twice, except when the edge is part of the interface;
	\item If the interface $\Gamma$ intersects the edges of an interface element at more than two points, then these points must lie on different edges of the elements;
	\item The interface $\Gamma$ within the interface element is $C^2$.
\end{itemize}

For $T \in \mathcal{T}_h$, define the diameter and area of $T$ as $h_T$ and $|T|$, respectively. Set $h = \max_{T \in \mathcal{T}_h} h_T$, which represents the maximum diameter of the partitions $\mathcal{T}_h$. Denote by $\mathcal{E}_h$ the set of all edges of the interface elements. Define $\mathcal{E}_h^n$ as the set of edges where the interface elements intersect the non-interface elements. This set includes all edges shared by both interface elements and non-interface elements. The set of edges intersecting the interface is defined as $\mathcal{E}_h^I$. For the integer $k \geqslant 1$, define $Q_k(T)$ as the space of tensor product polynomials on $T$ of degree up to $k$ in each variable. Denote by $P_k(e)$ the space of polynomials of degree up to $k$ on the edge $e \in \mathcal{E}_h$. 
\begin{figure}[H]
	\centering
	\setlength{\unitlength}{1bp}%
	\begin{picture}(150, 180)(0,0)
		\put(0,0){\includegraphics[scale=0.4]{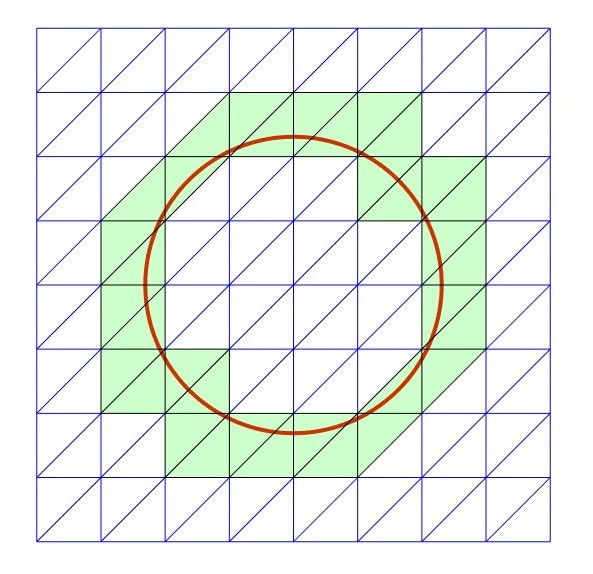}}
		\put(55.00,32.00){\fontsize{12.23}{15.07}\selectfont $\mathcal{T}_h^I$}
		\put(145.00,130.00){\fontsize{12.23}{15.07}\selectfont $\mathcal{T}_h^n$}	\put(132.00,100.00){\fontsize{12.23}{15.07}\selectfont $\Gamma$}
		\put(80.00,80.00){\fontsize{12.23}{15.07}\selectfont $\Omega_1$}
		\put(145.00,30.00){\fontsize{12.23}{15.07}\selectfont $\Omega_2$}
	\end{picture}%
	\caption{Unfitted meshes.}
	\label{unfitted_meshes_IEI}
\end{figure}

{The remainder of this section closely follows the presentation in \cite{Frenet_Serret_2024}. For the reader's convenience and ease of reference in subsequent sections, we reproduce the relevant material here.}

We assume that the parametrization of the interface $\Gamma$ is as follows:
$$
\bg(\xi)=(g_1(\xi),g_2(\xi))^T,\, [\xi_s, \xi_e] \rightarrow \Gamma,
$$ 
and is regular in the sense that $\bg'(\xi) \neq 0$ for all $\xi \in [\xi_s, \xi_e]$. In order to get the high order estimates, we assume $g_i(\xi) \in C^{k+2}(\xi_s,\xi_e)$, $i=1,\,2$. For the point on the interface $\Gamma$, the unit tangent vector $\bm{\tau}(\xi)$ is given by 
$$
\bm{\tau}(\xi)=\begin{pmatrix}
	\tau_1(\xi)\\
	\tau_2(\xi)
\end{pmatrix}=\frac{1}{\|\bg'(\xi)\|} \bg'(\xi),
$$
the unit normal vector is defined as
$$
\bn(\xi)=\begin{pmatrix}
	n_1(\xi)\\
	n_2(\xi)
\end{pmatrix}=Q\bm{\tau}(\xi),\quad Q=\begin{pmatrix}
	0 & 1\\
	-1 & 0
\end{pmatrix},
$$
and the signed curvature $\kappa(\xi)$ is denoted by
$$
\kappa(\xi)=\frac{g_1'(\xi)g_2''(\xi)-g_2'(\xi)g_1''(\xi)}{\|\bg'(\xi)\|^3}.
$$

In the neighborhood of the interface $\Gamma$, there exist curves that are locally parallel to $\Gamma$.  According to the Frenet apparatus, a curve parallel to $\Gamma$ with an offset distance $\eta$ has the following parametric form \cite{Frenet_Serret_2024}:
\begin{eqnarray*}
	\bx(\eta, \xi)=\begin{pmatrix}
		x(\eta, \xi)\\
		y(\eta, \xi)
	\end{pmatrix}
	=P_{\Gamma}(\eta, \xi) = \bg(\xi)+\eta \bn(\xi),\, \xi \in [\xi_s, \xi_e].
\end{eqnarray*} 

We assume that the interface $\Gamma$ has a tabular neighborhood with a half-band width $\varepsilon >0$, denoted by  
$$
N_{\Gamma}(\varepsilon)=P_{\Gamma}([-\varepsilon,\varepsilon] \times [\xi_s, \xi_e]).
$$
Obviously, for a mesh $\mathcal{T}_h$ with $h$, all the interface elements are inside the tubular neighborhood $N_{\Gamma}(h)$ of $\Gamma$. Hence, we assume that the mesh size $h$ is small enough such that $N_{\Gamma}(h) \subset N_{\Gamma}(\varepsilon)$ or simply $h < \varepsilon$. By \cite{Frenet_Serret_2024},  the mapping $P_{\Gamma}$ is bijective and has an inverse mapping defined as  $R_{\Gamma}: N_{\Gamma}(\varepsilon) \rightarrow [-\varepsilon,\varepsilon] \times [\xi_s, \xi_e]$, i.e.
$$
\begin{pmatrix}
	\eta(x,y)\\\xi(x,y)
\end{pmatrix}=R_{\Gamma}(x,y)=P^{-1}_{\Gamma}(x,y)\in [-\varepsilon,\varepsilon] \times [\xi_s, \xi_e], \, (x,y) \in N_{\Gamma}(\varepsilon).
$$

We start from recalling the well-known Frenet-Serret formulas \cite{gray2006modern} for the tangent vector and the normal vector:
\begin{eqnarray*}
	\bm{\tau}'(\xi)=-\kappa(\xi)\|\bg'(\xi)\|\bn(\xi),\quad
	\bn'(\xi)=\kappa(\xi)\|\bg'(\xi)\|\bm{\tau}(\xi).
\end{eqnarray*}
Using {these formulas}, the Jacobian matrix $\widehat{D}P_{\Gamma}(\eta, \xi)$ of the Frenet transformation $P_{\Gamma}(\eta, \xi)$ is {as follows:}
\begin{eqnarray*}
	\widehat{D}P_{\Gamma}(\eta, \xi)=\begin{pmatrix}
		\bn(\xi) & \bg'(\xi)+\eta \bn'(\xi)
	\end{pmatrix}=\begin{pmatrix}
		\bn(\xi) & \|\bg'(\xi)\|(1+\eta \kappa(\xi))\bm{\tau}(\xi)
	\end{pmatrix}.
\end{eqnarray*}
By the inverse function theorem, we have the following formula for the Jacobian of the inverse Frenet transformation $R_{\Gamma}(x,y)$:
\begin{eqnarray*}
	DR_{\Gamma}(x,y)=\begin{pmatrix}
		\bn(\xi)^T \\
		\|\bg'(\xi)\|^{-1} \psi(\eta,\xi)\bm{\tau}(\xi)^T
	\end{pmatrix}, \quad \psi(\eta,\xi)=(1+\eta \kappa(\xi))^{-1}.
\end{eqnarray*}

\section{The Immersed {\color{black}skeletal} Finite Element Space}
For the interface element $T \in \mathcal{T}_h^I$, we use the function $v_h=\{v_0, v_b\}$ to {discretize} the exact function. We construct the IFE function that strictly satisfy the interface conditions as the interior function $v_0$. The main idea of constructing this IFE function is found in \cite {Frenet_Serret_2024}. 

{For the interface element $T=\triangle P_1P_2P_3$, we construct a fictitious element $T_F$ to cover $T$. For each vertex $P_i$, $1 \leqslant i \leqslant 3$, there exists a unique $\xi_i \in [\xi_s, \xi_e]$ such that 
	\begin{align*}
		\|P_i - \bg(\xi_i)\|={\rm{dist}}(\Gamma, P_i).
	\end{align*}
	Therefore, we get two parameters $a_T$, $b_T \in [\xi_s, \xi_e]$ such that
	\begin{align*}
		a_T=\min(\xi_1,\xi_2,\xi_3), \quad b_T=\max(\xi_1,\xi_2,\xi_3).
	\end{align*}
	Then, we construct a rectangle $\widehat{T}_F=[-h,h] \times [a_T,b_T]$ and {curved trapezoid} $T_F=P_{\Gamma}(\widehat{T}_F)$. Define
	$$
	\widehat{T}_F^{1}=\{(\eta, \xi) \in \widehat{T}_F, \eta <0 \}\quad \text{and} \quad \widehat{T}_F^{2}=\{(\eta, \xi) \in \widehat{T}_F, \eta >0 \}
	$$
	as two subelements of the element $\widehat{T}_F$. For the interface element $T$, denote $\widehat{T}=R_{\Gamma}(T)$ and $\widehat{\Gamma}_{T_F}=R_{\Gamma}(\Gamma_{T_F})$, where $\Gamma_{T_F}=\Gamma \cap T_F$. It is worth noting that this transformation converts {interface segment} $\Gamma_{T_F}$ into {line segment} $\widehat{\Gamma}_{T_F}$, which directly avoids the geometric errors (see Figure \ref{inteface_element_reference_transform_IEI}).	
}
\begin{figure}
	\centering
	\setlength{\unitlength}{1bp}%
	\begin{picture}(365.01, 181.80)(0,0)
		\put(0,0){\includegraphics{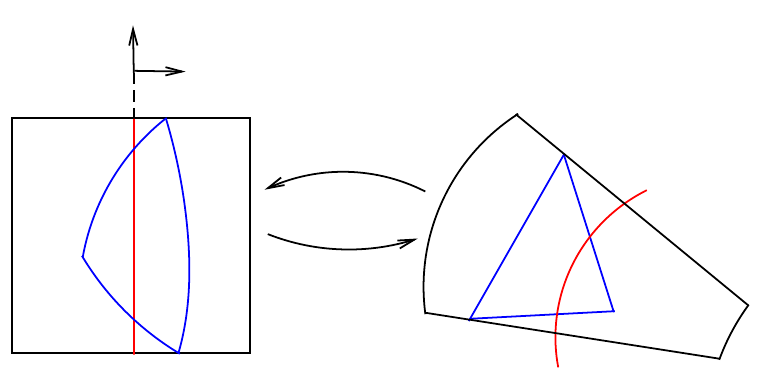}}
		\put(52.97,165.02){\fontsize{14.23}{17.07}\selectfont $\xi$}
		\put(89.39,143.58){\fontsize{14.23}{17.07}\selectfont $\eta$}
		\put(77.93,25.96){\fontsize{14.23}{17.07}\selectfont $\widehat{T}$}
		\put(11.26,108.91){\fontsize{14.23}{17.07}\selectfont $\widehat{T}^1_F$}
		\put(102.38,108.91){\fontsize{14.23}{17.07}\selectfont $\widehat{T}^2_F$}
		\put(102.95,19.40){\fontsize{14.23}{17.07}\selectfont $\widehat{T}_F$}
		\put(67.70,68.49){\fontsize{14.23}{17.07}\selectfont $\widehat{\Gamma}_{T_F}$}
		\put(257.64,52.37){\fontsize{14.23}{17.07}\selectfont $T$}
		\put(316.28,90.23){\fontsize{14.23}{17.07}\selectfont $\Gamma$}
		\put(327.79,19.33){\fontsize{14.23}{17.07}\selectfont $T_F$}
		\put(161.08,110.04){\fontsize{14.23}{17.07}\selectfont $R_{\Gamma}$}
		\put(157.41,45.99){\fontsize{14.23}{17.07}\selectfont $P_{\Gamma}$}
	\end{picture}%
	\caption{The transformation between the interface element  $T$  and $\widehat{T}$. }
	\label{inteface_element_reference_transform_IEI}
\end{figure}

We construct the IFE functions in $\widehat{T}_F$. First, {according to \cite {Frenet_Serret_2024}, we use} the mapping $R_{\Gamma}$ to  convert the interface conditions (\ref{Interface_condition_1})-(\ref {Laplacian_extend_interface_conditions}) to the conditions on $\widehat{\Gamma}_{T_F}$, i.e.
\begin{eqnarray}\label{inteface_condition_reference_IEI}
	\ljump \hat{u} \rjump_{\widehat{\Gamma}_{T_F}}=0,\quad 	\jump{\widehat{\beta}\dfrac{\partial \widehat{u}}{\partial \eta}}_{\widehat{\Gamma}_{T_F}}=0,\quad \jump{\widehat{\beta} \frac{\partial^j}{\partial \eta^j}\mathcal{L}(\widehat{u})}_{\widehat{\Gamma}_{T_F}}=0, \quad j=0,1,2,\cdots,k-2,
\end{eqnarray}
where
\begin{align*}
	& \widehat{u}(\eta, \xi)=u \circ P_{\Gamma}(\eta, \xi),\quad \widehat{\beta}=\beta \circ P_{\Gamma}(\eta, \xi),\\
	&\mathcal{L}(\widehat{u}(\eta, \xi))=\dfrac{\partial^2 \widehat{u}}{\partial \eta^2}(\eta, \xi)+J_0(\eta, \xi)\dfrac{\partial^2 \widehat{u}}{\partial \xi^2}(\eta, \xi)+J_1(\eta, \xi)\dfrac{\partial \widehat{u}}{\partial \eta}(\eta, \xi)+J_2(\eta, \xi)\dfrac{\partial \widehat{u}}{\partial \xi}(\eta, \xi),\\
	&J_0(\eta, \xi)=\left( \frac{\psi(\eta, \xi)}{\|\bg^{\prime}(\xi)\|} \right)^2,\quad J_1(\eta, \xi)=\kappa(\xi)\psi(\eta, \xi),\\
	&J_2(\eta, \xi)=-\left(\frac{\psi(\eta, \xi)}{\|\bg^{\prime}(\xi)\|} \right)^2\left(
	\eta \kappa^{\prime}(\xi)\psi(\eta, \xi)+\frac{\bg^{\prime}(\xi)\cdot \bg^{\prime\prime}(\xi)}{\|\bg^{\prime}(\xi)\|^2}
	\right).
\end{align*}
{For any integer $k \geqslant 1$, we use $	\widehat{\mathcal{V}}_k(\widehat{T}_F)$ to define the space of the IFE functions in $\widehat{T}_F$. According to \cite[Lemma 3]{Frenet_Serret_2024}, these piecewise functions 
	$$
	\widehat{\phi}_{i,j}(\eta, \xi)=\dfrac{1}{\widehat{\beta}}\eta^j p_i(\xi),\quad 0 \leqslant i \leqslant k, \, 1 \leqslant j \leqslant k,
	$$
	are in the space $	\widehat{\mathcal{V}}_k(\widehat{T}_F)$, where $\{p_i(\xi)\}_{i=0}^k$ is a basis of the space $P_k(\widehat{\Gamma}_{T_F})$. The $P_k(\widehat{\Gamma}_{T_F})$ denotes the space of polynomials in the variable $\xi$ of degree up to $k$. By \cite[Lemma 4]{Frenet_Serret_2024}, for $\widehat{\phi}_{i,0}|_{\widehat{T}_F^1}=\widehat{\phi}_{i,0}^1(\eta, \xi)=p_i(\xi), 0 \leqslant i \leqslant k$, there exists a unique $\widehat{\phi}_{i,0}^2(\eta, \xi) \in Q_k(\widehat{T}_F^2)$ such that $\widehat{\phi}_{i,0}(\eta, \xi) \in \widehat{\mathcal{V}}_k(\widehat{T}_F)$. Therefore, we can construct $(k+1)^2$ functions that precisely satisfy the interface conditions (\ref{Interface_condition_1})-(\ref {Laplacian_extend_interface_conditions}) to form a basis for the space $	\widehat{\mathcal{V}}_k(\widehat{T}_F)$.} The definition of the space $	\widehat{\mathcal{V}}_k(\widehat{T}_F)$ is as follows:
\begin{align*}
	\widehat{\mathcal{V}}_k(\widehat{T}_F)=\left\{ 
	\widehat{\phi}: \widehat{T}_F \rightarrow \mathbb{R}, {\widehat{\phi} \in Q_k(\widehat{T}_F)}  \,\text{and satisfies}\, \eqref{inteface_condition_reference_IEI}
	\right\}.
\end{align*}
Therefore, the space of these IFE functions in $T$ is defined as
\begin{align*}
	\mathcal{V}_k(T)=\left\{ 
	\widehat{\phi} \circ R_{\Gamma}:\widehat{\phi} \in  \widehat{\mathcal{V}}_k(\widehat{T}_F)
	\right\}.
\end{align*}
For the detail, readers can refer to \cite {Frenet_Serret_2024}. The definition of the boundary function $v_b$ is as follows:
\begin{align*}
	W_b(e)=&\{v_b \in  P_{k-1}(e) \,\text{on} \, e \in \mathcal{E}_h \setminus \mathcal{E}_h^I,\\
	&v_b=\{v_{1b},v_{2b}\}, v_{1b} \in  P_{k-1}(e_1),  v_{2b} \in  P_{k-1}(e_2) \,\text{on}\, e \in \mathcal{E}_h^I
	\},
\end{align*} 
where the definitions of $e_1$ and $e_2$ is as shown as Figure \ref{interface_element_IEI}.
\begin{figure}[H]
	\centering
	\setlength{\unitlength}{1bp}%
	\begin{picture}(126.55, 181.73)(0,0)
		\put(0,0){\includegraphics{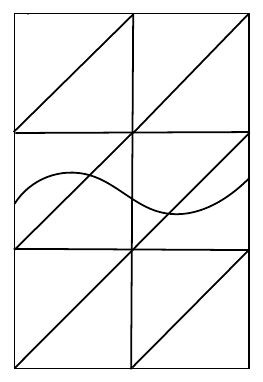}}
		\put(75.67,149.06){\fontsize{14.23}{17.07}\selectfont $\Omega_{1}$}
		\put(70.83,34.20){\fontsize{14.23}{17.07}\selectfont $\Omega_{2}$}
		\put(36.51,105.25){\fontsize{14.23}{17.07}\selectfont $e_1$}
		\put(32.12,72.91){\fontsize{14.23}{17.07}\selectfont $e_{2}$}		\put(85.12,67.91){\fontsize{14.23}{17.07}\selectfont $\Gamma$}
	\end{picture}
	\caption{The interface element $T \in \mathcal{T}_h^I$.}
	\label{interface_element_IEI}
\end{figure}
The {\color{black}skeletal finite element} space on the interface element $T \in \mathcal{T}_h^I$ is denoted by \begin{eqnarray*}
	W_h^I(T) = \{ v_h = \{v_0, v_b\}: v_0 \in \mathcal{V}_k(T),\, T \in \mathcal{T}_h^I,\, v_b \in W_b(e), e \in \mathcal{E}_h\}.
\end{eqnarray*}
Therefore, the global {\color{black}skeletal finite element} space is defined as follows:
\begin{eqnarray*}
	V_h^I=\{ v_h = \{v_0, v_b\}: v_h|_T \in W_h^I(T), \, \forall \, T \in \mathcal{T}_h^I \}.
\end{eqnarray*}

\section{The Immersed {\color{black}skeletal finite element} Scheme}
In this section, we apply the {\color{black}skeletal FEM} and the standard {\color{black}FEM} for the problems \eqref{Elliptic_problem}-\eqref{Laplacian_extend_interface_conditions} on unfitted meshes. We start by defining the finite element space and introducing the corresponding differential operators. Then, we present the proposed numerical scheme in detail.

For the non-interface elements, we define the following finite element space:
\begin{align*}
	V_h^n=\{ v_h \in H^1( \Omega \setminus \mathcal{T}_h^I): v_h|_T\in P_k(T), \, \forall \, T \in \mathcal{T}_h^n\}.
\end{align*}
For $T \in \mathcal{T}_h^n$, define $\Pi_h u$ as the Lagrange interpolate function of $u$ of degree up to $k$ in $V_h^n$. For $T \in \mathcal{T}_h^I$, denote by $Q_0$ the $L^2$ projection operator from $L^2(T)$ into $\mathcal{V}_k(T)$. On the edge $e \in \mathcal{E}_h$, we define the following $L^2$ projection operators:
{
	\begin{eqnarray*}
		Q_b u =\left\{\begin{array}{cl}
			Q_b^e u,  & \, e \in \mathcal{E}_h \setminus \mathcal{E}_h^I,\\
			\{Q_b^1 u_1, Q_b^2 u_2 \}, & \, e \in \mathcal{E}_h^I,\\
		\end{array}
		\right.
	\end{eqnarray*}
	and 
	\begin{eqnarray*}
		Q_\delta u =\left\{\begin{array}{cl}
			Q_b^e(\Pi_h u), &\, e \in \mathcal{E}_h^n,\\
			\{Q_b^1 u_1, Q_b^2 u_2 \},  &\, e \in \mathcal{E}_h^I,\\
			Q_b^eu,&\, e \in \mathcal{E}_h \setminus (\mathcal{E}_h^n \cup \mathcal{E}_h^I),
		\end{array}
		\right.
	\end{eqnarray*}
	where $Q_b^e$ is the $L^2$ projection operator from $L^2(e)$ into $P_{k-1}(e)$. For the interface edge $e=e_1 \cup e_2$, $Q_b^i$ is the $L^2$ projection operator from $L^2(e_i)$ into $P_{k-1}(e_i)$, $i=1,2$.}
By combining the space $V_h^n$ with the  {\color{black}skeletal finite element} space, the global finite element space is defined as follows:
\begin{align*}
	V_{h} &=\{  v_h: v_h|_{\mathcal{T}_h^n} \in V_h^n, v_h|_{\mathcal{T}_h^I} \in V_h^I, \, v_b =Q_b(v_h|_e) \, \text{on} \, e \in \mathcal{E}_h^n  \},\\
	V_{h}^0 &= \{ v_h: v_h \in V_{h}, v_h|_e =0 \,  \text{on} \, e \in \partial \Omega \}.
\end{align*}
Next, we present the definition of projection operators on the space $V_h$.  Set $Q_h u = \{ Q_0 u, Q_b u\}$ and $\widetilde{Q}_hu= \{ Q_0 u, Q_\delta u\}$. Denote by $P_hu$ and $I_hu$ the following projection operators:
\begin{eqnarray*}
	P_hu =\left\{\begin{array}{cc}
		u,& T \in \mathcal{T}_h^n,\\
		Q_h u, & T \in \mathcal{T}_h^I,\\
	\end{array}
	\right.
	\quad
	I_hu =\left\{\begin{array}{cc}
		\Pi_h u,& T \in \mathcal{T}_h^n,\\
		\widetilde{Q}_hu, & T \in \mathcal{T}_h^I.\\
	\end{array}
	\right.
\end{eqnarray*}

\begin{definition}
	For any $v \in V_h^I$, {\color{black}the gradient approximation} $\nabla_w v$ {\color{black} on each interface element $T \in \mathcal{T}_h^I$} satisfies
	\begin{eqnarray}\label{Def_weak_gradient_1}
		(\beta \nabla_w v, {\bf{q}})_T =(\beta\nabla v_0, {\bf{q}})_T -\langle Q_b v_0 -v_b, \beta {\bf{q}} \cdot \bn \rangle_{\partial T} ,\, \forall \, {\bf{q}} \in \nabla \mathcal{V}_k(T),
	\end{eqnarray}
	where $\bn$ is the unit outward normal vector on $\partial T$.
\end{definition}
Based on the above definitions, the following numerical scheme is proposed.

\begin{algorithm}[H]
	{\small
		\caption{The Immersed {\color{black}skeletal finite element} Scheme}
		Find $u_h \in V_{h}$ and $u_h = I_h g$ on $\partial \Omega$ to satisfy
		\begin{align}
			a^n(u_h,v_h)+a_s^I(u_h,v_h)=(f, v_h), \, \forall \, v_h \in V_{h}^0, \label{WGscheme}
		\end{align}
		where
		\begin{align*}
			a^n(u_h,v_h)&=\sum_{T \in \mathcal{T}_h^n}(\beta\nabla u_h, \nabla v_h)_T,\\
			s(u_h,v_h)&=\sum_{T \in \mathcal{T}_h^I} h_T^{-1}\langle \beta( Q_b u_0- u_b), Q_b v_0- v_b \rangle_{\partial T},\\
			a_s^I(u_h,v_h)&=\sum_{T \in \mathcal{T}_h^I}
			(\beta\nabla_w u_h, \nabla_w v_h)_T +s(u_h,v_h),\\
			(f,v_h)&=\sum_{T \in \mathcal{T}_h^n}(f,v)_T+\sum_{T \in \mathcal{T}_h^I}(f,v_0)_T.
		\end{align*}
	}
\end{algorithm}

{\color{black}In above numerical scheme, $s(u_h,v_h)$ is the LS stabilization to ensure the weak continuity of the numerical solution.}

We prove the well-posedness of the numerical scheme. Define the following semi-norm in $V_h+H^1(\Omega)$:
\begin{eqnarray}\label{H1norm}
	\begin{split}
		\| v_h \|_{1,h}^2= \sum_{T \in \mathcal{T}_h^n} \|\nabla v_h\|_T^2 +  \sum_{T \in \mathcal{T}_h^I}\left( \|\nabla_w v_h\|_T^2 +h_T^{-1}\| Q_b v_0- v_b\|_{\partial T}^2 \right).
	\end{split}
\end{eqnarray}

\begin{lemma}\label{Energynorm}
	$\| \cdot \|_{1,h}$ is a norm in $V_h^0$.
\end{lemma}
\begin{proof}
	Let $\| v_h \|_{1,h}=0$ for some $v_h \in V_{h}^0$. By the definition of $\| \cdot \|_{1,h}$, we have
	\begin{eqnarray*}
		\nabla v_h ={\bf{0}} \,\text{in}\, T \in \mathcal{T}_h^n ;\quad
		\nabla_w v_h ={\bf{0}}\,\text{in}\, T \in \mathcal{T}_h^I, \quad Q_b v_0 =v_b \, \text{on}\, e \in \mathcal{E}_h.
	\end{eqnarray*}
	Therefore, $v_h$ is a constant in $T \in \mathcal{T}_h^n$. Using $v_h=0$ on $\partial \Omega$, $v_h$ is zero in $T \in \mathcal{T}_h^n$.
	
	For $T \in \mathcal{T}_h^I$ and $w \in \mathcal{V}_k(T)$, according to Eq.\eqref{Def_weak_gradient_1}, we obtain
	\begin{eqnarray}\label{norm_proof_2}
		\begin{split}
			0=&(\beta \nabla_w v_h, \nabla w)_T\\
			=&(\beta  \nabla v_0, \nabla w)_T - \langle
			Q_b v_0 -v_b, \beta  \nabla w \cdot \bn \rangle_{\partial T}\\
			=&(\beta  \nabla v_0, \nabla w)_T.
		\end{split}
	\end{eqnarray}
	In Eq.(\ref{norm_proof_2}), taking $w = v_0$ yields $\nabla v_0 ={\bf{0}}$ in $T \in \mathcal{T}_h^I$. Therefore, $v_0$ is a constant in $T \in \mathcal{T}_h^I$. Combining with $Q_b v_0 = v_b$ on $e \in \mathcal{E}_h$, we get $v_b = Q_b v_0 = v_0$. By $v_b= Q_b(v_h|_e)=0 $ on $e \in \mathcal{E}_h^n$ and $v_h=0$ on $\partial \Omega$, we get $v_0 =v_b =0$. Therefore, $v_h =0$ is obtained. The proof of the lemma is completed.
\end{proof}

\begin{lemma}\label{Abounded_IEI}
	For any $ \bv,\bw \in V_h$, we have
	\begin{align*}
		|a_s^I(\bv,\bw)+a^n(\bv,\bw)|&\leqslant C \beta_2 \| \bv \|_{1,h} \cdot \| \bw \|_{1,h}.
	\end{align*}
\end{lemma}
According to the above lemma, it's easy to obtain the following theorem.

\begin{theorem}
	The numerical scheme (\ref{WGscheme}) has only one solution.	
\end{theorem}

\section{Error Equation}

In this section, we give the error equation and derive the error estimates under the $H^1$ norm and $L^2$ norm. Let $u_h$ denote the
numerical solution obtained from the scheme. The errors associated with $u$ are defined as follows:
$$
e_h = P_hu -u_h, \, \varepsilon_h= I_hu -u_h.
$$
For $T \in \mathcal{T}_h^I$, denote by $\mathcal{Q}_h$ the $L^2$ projection operator from $[L^2(T)]^2$ into $\nabla \mathcal{V}_k(T)$. In order to get the error estimate of $e_h$, we first give the error equation and error estimate of $\varepsilon_h$.
\begin{lemma}
	For $u \in H^1(\Omega)$, $T \in \mathcal{T}_h^I$ and ${\bf{q}} \in \nabla \mathcal{V}_k(T)$, the following properties hold true
	\begin{align}
		\begin{split}\label{weak_gradient_exchange_2}
			(\beta  \nabla_w (Q_h u), {\bf{q}})_T =(\beta \mathcal{Q}_h (\nabla u), {\bf{q}})_T + (\beta \nabla(Q_0 u -u),  {\bf{q}})_T
			- \langle Q_b (Q_0 u)- Q_b u,  \beta  {\bf{q}} \cdot \bn \rangle_{\partial T},
		\end{split}\\
		\begin{split}\label{weak_gradient_exchange_3}
			(\beta \nabla_w (\widetilde{Q}_h u), {\bf{q}})_T = (\beta  \mathcal{Q}_h (\nabla u), {\bf{q}})_T + (\beta \nabla(Q_0 u -u),  {\bf{q}})_T 
			- \langle Q_b (Q_0 u)- Q_{\delta} u,  \beta {\bf{q}} \cdot \bn \rangle_{\partial T}.
		\end{split}
	\end{align}
\end{lemma}
\begin{proof}
	For $T \in \mathcal{T}_h^I$ and ${\bf{q}} \in \nabla \mathcal{V}_k(T)$, it follows from Eq.\eqref{Def_weak_gradient_1} and the definition of the  $L^2$ projection operator $\mathcal{Q}_h$ that
	\begin{eqnarray*}
		\begin{split}
			(\beta \nabla_w (Q_h u), {\bf{q}})_T=&(\beta \nabla(Q_0 u), {\bf{q}})_T -\langle Q_b(Q_0 u) - Q_b u, \beta  {\bf{q}} \cdot \bn \rangle_{\partial T}\\
			=&(\beta \nabla u,{\bf{q}})_T+(\beta \nabla(Q_0 u -u), {\bf{q}})_T -\langle Q_b(Q_0 u) - Q_b u, \beta  {\bf{q}} \cdot \bn \rangle_{\partial T}\\
			=&(\beta \mathcal{Q}_h (\nabla u),  {\bf{q}})_T+(\beta \nabla(Q_0 u -u),  {\bf{q}})_T -\langle Q_b(Q_0 u) - Q_b u,  \beta {\bf{q}} \cdot \bn \rangle_{\partial T}.
		\end{split}
	\end{eqnarray*}
	Similarly, Eq.\eqref{weak_gradient_exchange_3} is obtained. The proof of the above lemma is completed.
\end{proof}

\begin{lemma}\label{error_equation_pre}
	For $u$ satisfying Eqs.(\ref{Elliptic_problem})-(\ref{Laplacian_extend_interface_conditions}) and $v \in V_h^0$, then we have the following equation
	\begin{eqnarray}\label{error_equation_1}
		\begin{split}
			\sum_{T \in \mathcal{T}_h^n}(\beta \nabla(\Pi_h u), \nabla v)_T+\sum_{T \in \mathcal{T}_h^I}(\beta \nabla_w \widetilde{Q}_h u,v)_T
			=\sum_{T \in \mathcal{T}_h^n}(f,v)_T+\sum_{T \in \mathcal{T}_h^I}(f,v_0)_T-\widetilde{\phi}_u(v), 
		\end{split}
	\end{eqnarray}
	where
	\begin{align*}
		\widetilde{\phi}_u(v)&=\ell_1(u,v)-\ell_2(u,v)+\widetilde{\ell}_3(u,v)-\ell_4(u,v)+\ell_5(u,v)-\ell_6(u,v),\\
		\ell_1(u,v)&=\sum_{T \in \mathcal{T}_h^I}\langle Q_bv_0 - v_b, \beta \mathcal{Q}_h (\nabla u) \cdot \bn - \beta \nabla u \cdot \bn \rangle_{\partial T},\\
		\ell_2(u,v)&=\sum_{T \in \mathcal{T}_h^I}(\nabla(Q_0 u- u), \beta \nabla_w v)_T,\\
		\widetilde{\ell}_3(u,v)&=\sum_{T \in \mathcal{T}_h^I} \langle
		Q_b(Q_0 u)-Q_{\delta} u, \beta\nabla_w v \cdot \bn \rangle_{\partial T},\\
		\ell_4(u,v)&=\sum_{T \in \mathcal{T}_h^I}\langle \beta\nabla u \cdot \bn, v_0 -Q_b v_0 \rangle_{\partial T},\\
		\ell_5(u,v)&=\sum_{e \in \mathcal{E}_h^n} \langle \beta\nabla u \cdot \bn_e, v- Q_b v\rangle_e,\\
		\ell_6(u,v)&=\sum_{T \in \mathcal{T}_h^n}(\beta \nabla(\Pi_h u-u),\nabla v)_T,
	\end{align*}
	and $\bn_e$ is the unit normal vector on $e \in \mathcal{E}_h^n$ pointing from the interface element $T \in \mathcal{T}_h^I$ into the non-interface element $T \in \mathcal{T}_h^n$.
\end{lemma}
\begin{proof}
	For the non-interface element $T \in \mathcal{T}_h^n$, integrating with the test function $v$ on two sides of  Eq.(\ref{Elliptic_problem}), applying integration by parts and
	\begin{align*}
		\sum_{T \in \mathcal{T}_h^n} \langle \beta\nabla u \cdot \bn, v  \rangle_{\partial T}=-\sum_{e \in \mathcal{E}_h^n} \langle \beta\nabla u \cdot \bn_e, v  \rangle_{e}
	\end{align*}
	leads to
	\begin{align}\label{proof_ee_9}
		&\sum_{T \in \mathcal{T}_h^n}(f,v)_T\nonumber\\
		=&\sum_{T \in \mathcal{T}_h^n}-(\nabla \cdot (\beta\nabla u),v)_T\\
		=&\sum_{T \in \mathcal{T}_h^n} (\beta\nabla u,  \nabla v)_T -\sum_{T \in \mathcal{T}_h^n} \langle \beta\nabla u \cdot \bn, v \rangle_{\partial T}\nonumber\\
		=&\sum_{T \in \mathcal{T}_h^n}(\beta \nabla(\Pi_h u),\nabla v)_T+\sum_{T \in \mathcal{T}_h^n} (\beta\nabla(u-\Pi_h u),  \nabla v)_T +\sum_{e \in \mathcal{E}_h^n} \langle \beta\nabla u \cdot \bn_e, v \rangle_{e}.\nonumber
	\end{align}
	Similarly, for $T \in \mathcal{T}_h^I$, integrating with $v_0$ of $v=\{v_0, v_b\} \in V_h^0$ on two sides of Eq.(\ref{Elliptic_problem}) yields 
	\begin{eqnarray}\label{proof_ee_3}
		\sum_{T \in \mathcal{T}_h^I}(-\nabla \cdot (\beta \nabla u), v_0)_T=\sum_{T \in \mathcal{T}_h^I}(f,v_0)_T.
	\end{eqnarray}
	Using integration by parts leads to
	\begin{eqnarray}\label{proof_ee_4}
		\sum_{T \in \mathcal{T}_h^I} (\beta \nabla u, \nabla v_0)_T - \sum_{T \in \mathcal{T}_h^I}  \langle
		\beta \nabla u \cdot \bn, v_0 \rangle_{\partial T}=\sum_{T \in \mathcal{T}_h^I} (f,v_0)_T.
	\end{eqnarray}
	We use the fact that $\sum_{T \in \mathcal{T}_h^I} \langle
	\beta \nabla u \cdot \bn, v_b \rangle_{\partial T} =\sum_{e \in \mathcal{E}_h^n} \langle
	\beta \nabla u \cdot \bn_e, v_b \rangle_{e}$ to derive
	\begin{eqnarray}\label{proof_ee_5}
		\begin{split}
			&\sum_{T \in \mathcal{T}_h^I} \langle
			\beta \nabla u \cdot \bn, v_0 \rangle_{\partial T}\\
			=&\sum_{T \in \mathcal{T}_h^I} \langle
			\beta \nabla u \cdot \bn, v_0 -Q_b v_0\rangle_{\partial T}
			+\sum_{T \in \mathcal{T}_h^I} \langle
			\beta \nabla u \cdot \bn, Q_b v_0-v_b\rangle_{\partial T}+\sum_{e \in \mathcal{E}_h^n} \langle
			\beta \nabla u \cdot \bn_e, v_b \rangle_{e}.
		\end{split}
	\end{eqnarray}
	Substituting Eq.(\ref{proof_ee_5}) into Eq.(\ref{proof_ee_4}) yields
	\begin{eqnarray}
		\begin{split}\label{proof_ee_10}
			&\sum_{T \in \mathcal{T}_h^I}(\beta \nabla u, \nabla v_0)_T\\
			=&\sum_{T \in \mathcal{T}_h^I}(f,v_0)_T+\sum_{T \in \mathcal{T}_h^I} \langle
			\beta \nabla u \cdot \bn, v_0 -Q_b v_0\rangle_{\partial T}
			+\sum_{T \in \mathcal{T}_h^I} \langle
			\beta \nabla u \cdot \bn, Q_b v_0-v_b\rangle_{\partial T}\\
			&+\sum_{e \in \mathcal{E}_h^n} \langle
			\beta \nabla u \cdot \bn_e, v_b \rangle_{e}.
		\end{split}
	\end{eqnarray}
	For the interface element $T \in \mathcal{T}_h^I$, using Eq.(\ref{weak_gradient_exchange_2}), Eq.(\ref{Def_weak_gradient_1}) and the definition of the $L^2$ projection operator $\mathcal{Q}_h$, we obtain
	\begin{eqnarray}\label{proof_ee_2}
		\begin{split}
			&(\beta\nabla_w \widetilde{Q}_h u, \nabla_w v)_T\\
			=&(\beta\mathcal{Q}_h (\nabla u), \nabla_w v)_T+(\nabla(Q_0 u -u), \beta \nabla_w v)_T - \langle Q_b (Q_0 u)- Q_{\delta} u, \beta \nabla_w v \cdot \bn \rangle_{\partial T}\\
			=&(\beta\mathcal{Q}_h (\nabla u), \nabla v_0)_T -\langle Q_b v_0 -v_b, \beta\mathcal{Q}_h  (\nabla u)\cdot \bn \rangle_{\partial T}\\
			&+(\nabla(Q_0 u -u), \beta \nabla_w v)_T - \langle Q_b (Q_0 u)- Q_{\delta} u, \beta \nabla_w v \cdot \bn \rangle_{\partial T}\\
			=&(\beta \nabla u, \nabla v_0)_T -\langle Q_b v_0 -v_b, \beta\mathcal{Q}_h  (\nabla u) \cdot \bn \rangle_{\partial T}\\
			&+(\nabla(Q_0 u -u), \beta \nabla_w v)_T - \langle Q_b (Q_0 u)- Q_{\delta} u, \beta \nabla_w v \cdot \bn \rangle_{\partial T}.
		\end{split}
	\end{eqnarray}
	By using the fact that $v_b = Q_b v$ on $e \in \mathcal{E}_h^n$ and Eqs.(\ref{proof_ee_9})-(\ref{proof_ee_2}), we have
	\begin{align}\label{proof_ee_6}
		&\sum_{T \in \mathcal{T}_h^n}(\beta \nabla(\Pi_h u),\nabla v)_T +
		\sum_{T \in \mathcal{T}_h^I} (\beta\nabla_w \widetilde{Q}_h u, \nabla_w v)_T \nonumber\\
		= &\sum_{T \in \mathcal{T}_h^n}(f,v)_T+\sum_{T \in \mathcal{T}_h^I}(f,v_0)_T-\sum_{T \in \mathcal{T}_h^I} \langle Q_b v_0 -v_b, \beta \mathcal{Q}_h  (\nabla u) \cdot \bn - \beta \nabla u \cdot \bn \rangle_{\partial T}\nonumber\\
		&+\sum_{T \in \mathcal{T}_h^I}(\nabla(Q_0 u-u), \beta \nabla_w v)_T-\sum_{T \in \mathcal{T}_h^I} \langle
		Q_b(Q_0 u)-Q_{\delta} u, \beta\nabla_w v \cdot \bn \rangle_{\partial T}\\
		&+\sum_{T \in \mathcal{T}_h^I}\langle \beta\nabla u \cdot \bn, v_0 -Q_b v_0 \rangle_{\partial T}-\sum_{e \in \mathcal{E}_h^n}\langle \beta \nabla u \cdot \bn_e, v-Q_b v \rangle_{\partial T}\nonumber\\
		&+\sum_{T \in \mathcal{T}_h^n} (\beta\nabla(\Pi_h u-u),  \nabla v)_T .\nonumber
	\end{align}
	The proof of the above lemma is completed.
\end{proof}

\begin{lemma}\label{error_equation_pre_varepsilon_h}
	The error $\varepsilon_h$ satisfies the following error equation
	\begin{eqnarray}\label{error_equation_varepsilon_h}
		a^n(\varepsilon_h,v)+a_s^I(\varepsilon_h,v)=s(\widetilde{Q}_hu,v)-\widetilde{\phi}_u(v), \quad \forall \, v \in V_h^0.
	\end{eqnarray}
\end{lemma}
\begin{proof}
	For any $v \in V_h^0$, using Lemma \ref{error_equation_pre} and adding the stabilizer $s(\widetilde{Q}_hu,v)$ leads to 
	\begin{eqnarray*}\label{proof_ee_7_IEI}
		a^n(\Pi_hu,v)+a_s^I(\widetilde{Q}_hu,v)=\sum_{T \in \mathcal{T}_h^n}(f,v)_T+\sum_{T \in \mathcal{T}_h^I}(f,v_0)_T+s(\widetilde{Q}_hu,v)-\widetilde{\phi}_u(v).
	\end{eqnarray*}
	Subtracting Eq.(\ref{WGscheme}) from the above equation yields 
	\begin{eqnarray}\label{proof_ee_8_IEI}
		a^n(\varepsilon_h,v)+a_s^I(\varepsilon_h,v)=s(\widetilde{Q}_hu,v)-\widetilde{\phi}_u(v).
	\end{eqnarray}
	The proof of the above lemma is completed.
\end{proof}

Similar to the proof of lemma \ref{error_equation_pre}-\ref{error_equation_pre_varepsilon_h}, using Eq.\eqref{weak_gradient_exchange_2} leads to the following theorem.

\begin{theorem}
	The error $e_h$ satisfies the following error equation
	\begin{eqnarray}\label{error_equation}
		a^n(e_h,v)+a_s^I(e_h,v)=s(Q_hu,v)-\phi_u(v), \quad \forall \, v \in V_h^0,
	\end{eqnarray}
	where 
	\begin{align*}
		\phi_u(v)&=\ell_1(u,v)-\ell_2(u,v)+\ell_3(u,v)-\ell_4(u,v)+\ell_5(u,v),\\
		\ell_3(u,v)&=\sum_{T \in \mathcal{T}_h^I} \langle
		Q_b(Q_0 u)-Q_b u, \beta\nabla_w v \cdot \bn \rangle_{\partial T}.
	\end{align*}
\end{theorem}

\section{Error Estimate in $H^1$ norm}

\begin{lemma}\label{H1error_varepsilon_h}
	Assume $u \in PH^{k+1}(\Omega)$ is the exact solution of
	the Eqs.(\ref{Elliptic_problem})-(\ref{Laplacian_extend_interface_conditions}) and $u_h \in V_h$ is the numerical solution obtained from the scheme (\ref{WGscheme}), then we have
	\begin{eqnarray}\label{H1errorestimates_varepsilon_h}
		\| \varepsilon_h \|_{1,h} \leqslant C \left(\dfrac{\beta_2}{\beta_1}\right)^2 h^{k}(\|u_1\|_{k+1,\Omega_1}+\|u_2\|_{k+1,\Omega_2}).
	\end{eqnarray}	
\end{lemma}
\begin{proof}
	Choosing $v = \varepsilon_h \in V_h^0$ in Eq.(\ref{error_equation_varepsilon_h}) and using Lemma \ref{H1estimate}  leads to 
	\begin{eqnarray}\label{H1error_varepsilon_h_proof_10_IEI}
		\begin{split}
			\beta_1\|\varepsilon_h\|_{1,h}^2 \leqslant&a^n(\varepsilon_h,\varepsilon_h)+a_s^I(\varepsilon_h,\varepsilon_h)\\
			\leqslant& C\beta_2 h^k (\|u_1\|_{k+1,\Omega_1}+\|u_2\|_{k+1,\Omega_2}) \|\varepsilon_h\|_{1,h}\\
			&+C\frac{\beta_2^2}{\beta_1} h^k (\|u_1\|_{k+1,\Omega_1}+\|u_2\|_{k+1,\Omega_2}) \|\varepsilon_h\|_{1,h}.
		\end{split}
	\end{eqnarray}
	Therefore, we have
	\begin{eqnarray}
		\| \varepsilon_h \|_{1,h} \leqslant C \left(\dfrac{\beta_2}{\beta_1}\right)^2 h^{k}(\|u_1\|_{k+1,\Omega_1}+\|u_2\|_{k+1,\Omega_2}).
	\end{eqnarray}
	The proof of the above lemma is completed.
\end{proof}

\begin{theorem}\label{H1error}
	Based on the assumption in Lemma \ref{H1error_varepsilon_h}, we derive the following error estimate
	\begin{eqnarray}\label{H1errorestimates}
		\| e_h \|_{1,h} \leqslant C \left(\dfrac{\beta_2}{\beta_1}\right)^2h^k (\|u_1\|_{k+1,\Omega_1}+\|u_2\|_{k+1,\Omega_2}).
	\end{eqnarray}
\end{theorem}
\begin{proof}
	Choosing $v = \varepsilon_h \in V_h^0$ in Eq.(\ref{error_equation}) leads to 
	\begin{eqnarray*}
		a^n(e_h,\varepsilon_h)+a_s^I(e_h,\varepsilon_h)=s(Q_hu,\varepsilon_h)-\phi_u(\varepsilon_h).
	\end{eqnarray*}
	Thus, we have
	\begin{eqnarray*}
		a^n(e_h,e_h)+a_s^I(e_h,e_h)=s(Q_hu,\varepsilon_h)-\phi_u(\varepsilon_h)-a^n(e_h,I_h u - P_hu)-a_s^I(e_h, I_h u - P_hu).
	\end{eqnarray*}
	For $a^n(e_h,I_h u - P_hu)$, by the Cauchy-Schwarz inequality and the {interpolation estimate}, we have
	\begin{eqnarray}\label{H1_proof_1}
		\begin{split}
			|a^n(e_h,I_h u - P_hu)|=&\left|\sum_{T \in \mathcal{T}_h^n}(\beta \nabla e_h, \nabla (\Pi_h u -u ))_T\right|\\
			\leqslant & C \beta_2 \left(\sum_{T \in \mathcal{T}_h^n} \|\nabla e_h\|_T^2  \right)^{\frac{1}{2}} \left(\sum_{T \in \mathcal{T}_h^n} \|\nabla (\Pi_h u -u )\|_T^2  \right)^{\frac{1}{2}}\\
			\leqslant & C\beta_2 h^k (\|u_1\|_{k+1,\Omega_1}+\|u_2\|_{k+1,\Omega_2}) \|e_h\|_{1,h}.
		\end{split}
	\end{eqnarray}
	For $a_s^I(e_h,I_h u - P_hu)$, since $\widetilde{Q}_h  u-Q_hu=\{0,Q_{\delta} u - Q_b u\}$, we get
	\begin{align}
		&\left|a_s^I(e_h,I_h u - P_hu)\right|
		=\left|a_s^I(e_h,\widetilde{Q}_h  u-Q_hu)\right|\nonumber\\
		=&\Bigg|\sum_{T \in \mathcal{T}_h^I} (\beta \nabla_w e_h, \nabla_w (\widetilde{Q}_h  u-Q_hu))_T+ \sum_{T \in \mathcal{T}_h^I} h_T^{-1} \langle \beta (Q_b e_0-e_b), Q_b u -Q_b(\Pi_h u) \rangle_{\partial T \cap \mathcal{E}_h^n} \Bigg|\nonumber\\
		\leqslant &C\beta_2 \sum_{T \in \mathcal{T}_h^I} \|\nabla_w e_h\|_T \|\nabla_w (\widetilde{Q}_h  u-Q_hu)\|_T\nonumber \\
		&+ C\beta_2 \sum_{T \in \mathcal{T}_h^I} h_T^{-1} \|Q_b e_0-e_b\|_{\partial T \cap \mathcal{E}_h^n} \|Q_b u -Q_b(\Pi_h u) \|_{\partial T \cap \mathcal{E}_h^n}\\
		\leqslant & C\beta_2 \left(\sum_{T \in \mathcal{T}_h^I}\|\nabla_w e_h\|_T^2
		\right)^{\frac{1}{2}} \left(\sum_{T \in \mathcal{T}_h^I}\|\nabla_w (\widetilde{Q}_h  u-Q_hu)\|_T^2
		\right)^{\frac{1}{2}}\nonumber\\
		&+C\beta_2 \left(\sum_{T \in \mathcal{T}_h^I}h_T^{-1} \|Q_b e_0-e_b\|_{\partial T \cap \mathcal{E}_h^n}^2 \right)^{\frac{1}{2}}\left(\sum_{T \in \mathcal{T}_h^I}h_T^{-1} \|Q_b u -Q_b(\Pi_h u)\|_{\partial T \cap \mathcal{E}_h^n}^2 \right)^{\frac{1}{2}}\nonumber\\
		\leqslant & C\beta_2\|e_h\|_{1,h}\left( \sum_{T \in \mathcal{T}_h^I} \|\nabla_w (\widetilde{Q}_h  u-Q_hu)\|_T^2+h_T^{-1}\|\Pi_h u -u \|^2_{\partial T \cap \mathcal{E}_h^n} \right)^{\frac{1}{2}}\nonumber\\
		\leqslant & C\beta_2\|e_h\|_{1,h}\left( \sum_{T \in \mathcal{T}_h^I} \|\nabla_w (\widetilde{Q}_h  u-Q_hu)\|_T^2+h^{2k}(\|u_1\|_{k+1,\Omega_1}^2+\|u_2\|_{k+1,\Omega_2}^2)
		\right)^{\frac{1}{2}}.\nonumber
	\end{align}
	For $\|\nabla_w (\widetilde{Q}_h  u-Q_hu)\|_T$ and $\mathbf{q} \in \nabla \mathcal{V}_k(T)$, it follows from Eq.\eqref{Def_weak_gradient_1}, the Cauchy-Schwarz inequality, the definition of $L^2$ projection operator $Q_b$ and Lemma \ref{inverse_interface_element} that
	\begin{eqnarray*}
		\begin{split}
			&(\beta\nabla_w (\widetilde{Q}_h  u-Q_hu), \mathbf{q})_T\\
			=&\langle Q_b(\Pi_hu ) -Q_bu, \beta \mathbf{q} \cdot \bn \rangle_{\partial T \cap \mathcal{E}_h^n}\\
			\leqslant & C \beta_2 h_T^{-\frac{1}{2}}\|\mathbf{q}\|_T \|u - \Pi_h u\|_{\partial T \cap \mathcal{E}_h^n}\\
			\leqslant &C \beta_2  h^k(\|u_1\|_{k+1,\Omega_1}+\|u_2\|_{k+1,\Omega_2}) \|\mathbf{q}\|_T.
		\end{split}
	\end{eqnarray*}
	Taking $\mathbf{q}=\nabla_w (\widetilde{Q}_h  u-Q_hu)$ in the above estimate, we get
	\begin{eqnarray*}
		\|\nabla_w (\widetilde{Q}_h  u-Q_hu)\|_T \leqslant  C  \frac{\beta_2}{\beta_1} h^k(\|u_1\|_{k+1,\Omega_1}+\|u_2\|_{k+1,\Omega_2}) .
	\end{eqnarray*}
	Thus, we obtain
	\begin{eqnarray}\label{H1_proof_2}
		\left|a_s^I(e_h,\widetilde{Q}_h  u-Q_hu)\right| \leqslant C \frac{\beta_2^2}{\beta_1} h^k\|u\|_{k+1} \|e_h\|_{1,h}.
	\end{eqnarray}
	According to the estimates (\ref{H1estimate1})-(\ref{H1estimate5}), (\ref{H1_proof_1})-(\ref{H1_proof_2}), Lemma \ref{H1error_varepsilon_h} and Young's inequality, we have
	\begin{align*}
		&\beta_1\| e_h \|_{1,h}^2 \\
		\leqslant& C\beta_2  h^k(\|u_1\|_{k+1,\Omega_1}+\|u_2\|_{k+1,\Omega_2}) ( \| e_h \|_{1,h}+ \| \varepsilon_h \|_{1,h} )\\
		&+C \frac{\beta_2^2}{\beta_1} h^k(\|u_1\|_{k+1,\Omega_1}+\|u_2\|_{k+1,\Omega_2}) ( \| e_h \|_{1,h}+ \| \varepsilon_h \|_{1,h} )\\
		\leqslant& C \frac{\beta_2^2}{\beta_1} h^k(\|u_1\|_{k+1,\Omega_1}+\|u_2\|_{k+1,\Omega_2})  \| e_h \|_{1,h}
		+C\frac{\beta_2^4}{\beta_1^3} h^{2k}(\|u_1\|_{k+1,\Omega_1}+\|u_2\|_{k+1,\Omega_2})^2 \\
		\leqslant & C\frac{\beta_2^4}{\beta_1^3} h^{2k}(\|u_1\|_{k+1,\Omega_1}+\|u_2\|_{k+1,\Omega_2})^2 +\frac{\beta_1}{2}\| e_h \|_{1,h}^2.
	\end{align*}
	Therefore, the following estimate holds true
	$$
	\| e_h \|_{1,h} \leqslant C\left(\dfrac{\beta_2}{\beta_1}\right)^2 h^{k}(\|u_1\|_{k+1,\Omega_1}+\|u_2\|_{k+1,\Omega_2}).
	$$
	The proof of the above theorem is completed. 
\end{proof}

\section{Error Estimate in the $L^2$ norm}
In this section, we use the dual argument to give the error estimate in the $L^2$ norm. It's obvious that error $e_0$  satisfies
$$
e_0=\varepsilon_0+\left\{ \begin{array}{cc}
	u-\Pi_h u, & T \in \mathcal{T}_h^n,\\
	0, & T \in \mathcal{T}_h^I.
\end{array}
\right.
$$
It's easy to get the estimate of $u-\Pi_h u$.
Thus, we only need to consider the estimate of $\varepsilon_0$ in the $L^2$ norm. We consider the following problem: seeking $\varphi$ to satisfy
\begin{align}
	-\nabla \cdot (\beta \nabla \varphi) &=\varepsilon_0,\, \text{in} \, \Omega,\label{dualproblem1}\\
	\varphi &= 0, \,\,\, \text{on} \, \partial \Omega,\label{dualproblem2}\\
	\varphi_1 - \varphi_2 &= 0, \,\,\,\text{on}\, \Gamma,\label{dualproblem3}\\
	\beta_1 \nabla \varphi_1 \cdot \bn - \beta_2 \nabla \varphi_2 \cdot \bn &= 0, \,\,\, \text{on}\, \Gamma,\label{dualproblem4}
\end{align}
where 
$$\varepsilon_0=\left\{ \begin{array}{cc}
	\Pi_h u-u_h, & T \in \mathcal{T}_h^n,\\
	Q_0u -u_0, & T \in \mathcal{T}_h^I.
\end{array}
\right.$$
Assume the solution $\varphi$ satisfies $H^2$-regularity \cite{Enrich_nonhomogeneous_2023}, i.e
\begin{eqnarray}\label{H2regularity_1}
	{	\|\beta \varphi\|_{2,\Omega} \leqslant C \|\varepsilon_0\|.}
\end{eqnarray}
Thus, we obtain
\begin{eqnarray}\label{H2regularity}
	{	\| \varphi\|_{2,\Omega} \leqslant \frac{C}{\beta_1}  \|\varepsilon_0\|.}
\end{eqnarray}
\begin{theorem}
	Based on the assumption of Theorem \ref{H1error}, we have the following error estimate in the $L^2$ norm
	\begin{eqnarray}
		\|e_0\| \leqslant C \left({\frac{\beta_2}{\beta_1}}\right)^4 h^{k+1} (\|u_1\|_{k+1,\Omega_1}+\|u_2\|_{k+1,\Omega_2}).
	\end{eqnarray}
\end{theorem}
\begin{proof}
	Using the triangular inequality leads to 
	\begin{align*}
		\|e_0\|\leqslant \sum_{T \in \mathcal{T}_h^n} \|u-\Pi_h u\|_T+\|\varepsilon_0\|\leqslant C h^{k+1}(\|u_1\|_{k+1,\Omega_1}+\|u_2\|_{k+1,\Omega_2})+\|\varepsilon_0\|.
	\end{align*}
	Next, we give the estimate of $\varepsilon_0$. Taking $v=\varepsilon_h$ in Eq.\eqref{error_equation_1} yields 
	\begin{eqnarray}\label{proof_L2_1_IEI}
		\begin{split}
			\sum_{T \in \mathcal{T}_h^n}(\beta \nabla(\Pi_h \varphi), \nabla \varepsilon_h)_T+\sum_{T \in \mathcal{T}_h^I}(\beta \nabla_w \widetilde{Q}_h \varphi,\varepsilon_h)_T=(\varepsilon_0,\varepsilon_0)-\widetilde{\phi}_{\varphi}(\varepsilon_h).
		\end{split}
	\end{eqnarray}
	Thus, we have
	\begin{align}\label{proof_L2_2_IEI}
		\begin{split}
			\|\varepsilon_0\|^2=a^n(\Pi_h \varphi, \varepsilon_h)+a_s^I(\widetilde{Q}_h \varphi,\varepsilon_h)-s(\widetilde{Q}_h \varphi,\varepsilon_h)+\widetilde{\phi}_{\varphi}(\varepsilon_h).
		\end{split}
	\end{align}
	Choosing $v=I_h \varphi \in V_h^0$ in Eq.(\ref{error_equation_varepsilon_h}) leads to
	\begin{eqnarray}\label{proof_L2_23_IEI}
		a^n(\varepsilon_h,\Pi_h\varphi)+a_s^I(\varepsilon_h,\widetilde{Q}_h  \varphi)=s( \widetilde{Q}_h u,\widetilde{Q}_h \varphi)-\widetilde{\phi}_u(I_h \varphi).
	\end{eqnarray}	
	Therefore, we get
	\begin{eqnarray}\label{proof_L2_25}
		\begin{split}
			\begin{split}
				\|\varepsilon_0\|^2=&s( \widetilde{Q}_h u,\widetilde{Q}_h \varphi)-\widetilde{\phi}_u(I_h \varphi)-s(\widetilde{Q}_h \varphi,\varepsilon_h)+\widetilde{\phi}_{\varphi}(\varepsilon_h).
			\end{split}
		\end{split}
	\end{eqnarray}
	By Lemma \ref{H1estimate}, we obtain
	\begin{align}\label{proof_L2_4}
		\begin{split}
			|-s(\widetilde{Q}_h \varphi,\varepsilon_h)+\widetilde{\phi}_{\varphi}(\varepsilon_h)|\leqslant C  \frac{\beta_2^2}{\beta_1} h \|\varphi\|_{2} \|\varepsilon_h\|_{1,h}.
		\end{split}
	\end{align}
	For $s(\widetilde{Q}_h u, \widetilde{Q}_h \varphi)$, according to the Cauchy-Schwarz inequality, the definition of the $L^2$ projection operator $Q_b$, the trace inequality, and the projection inequality \eqref{projectorestimate1}, we get
	\begin{align}\label{proof_L2_5}
		&|s(\widetilde{Q}_h u, \widetilde{Q}_h \varphi)|\nonumber\\
		\leqslant&\left|\sum_{T \in \mathcal{T}_h^I} h_T^{-1}\langle \beta(Q_b(Q_0 u)-Q_b u),  Q_b(Q_0 \varphi)-Q_b \varphi \rangle_{\partial T \cap \mathcal{E}_h^I}  \right|\nonumber\\
		&+\left|\sum_{T \in \mathcal{T}_h^I} h_T^{-1}\langle \beta (Q_b(Q_0 u)-Q_b (\Pi_h u)),  Q_b(Q_0 \varphi)-Q_b (\Pi_h \varphi) \rangle_{\partial T \cap \mathcal{E}_h^n}  \right|\nonumber\\
		\leqslant& C \beta_2 \left( \sum_{T \in \mathcal{T}_h^I} h_T^{-1}\|Q_0 u-u \|_{\partial T \cap \mathcal{E}_h^I}^2  \right)^{\frac{1}{2}} \left( \sum_{T \in \mathcal{T}_h^I} h_T^{-1}\|Q_0 \varphi - \varphi \|_{\partial T \cap \mathcal{E}_h^I}^2 \right)^{\frac{1}{2}}\\
		&+ C \beta_2 \left( \sum_{T \in \mathcal{T}_h^I} h_T^{-1}\|Q_0 u-u \|_{\partial T \cap \mathcal{E}_h^n}^2+h_T^{-1}\|\Pi_h u - u \|_{\partial T \cap \mathcal{E}_h^n}^2  \right)^{\frac{1}{2}}\nonumber\\
		&\left( \sum_{T \in \mathcal{T}_h^I} h_T^{-1}\|Q_0 \varphi - \varphi \|_{\partial T \cap \mathcal{E}_h^n}^2+h_T^{-1}\|\Pi_h \varphi - \varphi \|_{\partial T \cap \mathcal{E}_h^n}^2 \right)^{\frac{1}{2}}\nonumber\\
		\leqslant & C \beta_2 h^{k+1} (\|u_1\|_{k+1,\Omega_1}+\|u_2\|_{k+1,\Omega_2}) \|\varphi\|_{2}.\nonumber
	\end{align}
	Next, we estimate each of the term $\widetilde{\phi}_u(I_h \varphi)$.
	
	(1) For $\ell_1(u,\widetilde{Q}_h \varphi)$, we have the following estimate:
	\begin{align}\label{proof_L2_6}
		|\ell_1(u,\widetilde{Q}_h \varphi)|
		\leqslant&\left|\sum_{T \in \mathcal{T}_h^I} \langle Q_b(Q_0 \varphi)-Q_b \varphi, \beta \mathcal{Q}_h(\nabla u) \cdot \bn - \beta \nabla u \cdot \bn \rangle_{\partial T \cap \mathcal{E}_h^I }
		\right|\nonumber\\
		&+\left|\sum_{T \in \mathcal{T}_h^I} \langle Q_b(Q_0 \varphi)-Q_b (\Pi_h \varphi), \beta \mathcal{Q}_h(\nabla u)  \cdot \bn - \beta \nabla u \cdot \bn \rangle_{\partial T \cap \mathcal{E}_h^n}
		\right|\\
		\leqslant &C  \left(\sum_{T \in \mathcal{T}_h^I}\|Q_0 \varphi - \varphi \|_{\partial T }^2+\sum_{T \in \mathcal{T}_h^I}\|\Pi_h \varphi - \varphi \|_{\partial T \cap \mathcal{E}_h^n}^2  \right)^{\frac{1}{2}}\nonumber\\
		&\quad \left(\sum_{T \in \mathcal{T}_h^I} \| \beta \mathcal{Q}_h(\nabla u)  \cdot \bn - \beta \nabla u \cdot \bn \|_{\partial T}^2  \right)^{\frac{1}{2}}\nonumber\\
		\leqslant& C\beta_2  h^{k+1}(\|u_1\|_{k+1,\Omega_1}+\|u_2\|_{k+1,\Omega_2})\|\varphi\|_2.\nonumber
	\end{align}
	(2) For $\ell_4(u, \widetilde{Q}_h \varphi)$, using the fact that $\sum_{T \in \mathcal{T}_h^I} \langle Q_b(Q_0 \varphi)- Q_0 \varphi, Q_b(\beta\nabla u \cdot \bn) \rangle_{\partial T}=0$ leads to
	\begin{align}\label{proof_L2_7}
		&|\ell_4(u, \widetilde{Q}_h \varphi)|\nonumber\\ =&\left|\sum_{T \in \mathcal{T}_h^I} \langle Q_b(Q_0 \varphi)- Q_0 \varphi, \beta\nabla u \cdot \bn \rangle_{\partial T} \right|\nonumber\\
		=&\left|\sum_{T \in \mathcal{T}_h^I} \langle Q_b(Q_0 \varphi)- Q_0\varphi, \beta\nabla u \cdot \bn - Q_b(\beta\nabla u \cdot \bn)  \rangle_{\partial T} \right|\nonumber\\
		=&\left|\sum_{T \in \mathcal{T}_h^I} \langle Q_b(Q_0 \varphi)- Q_b\varphi+Q_b\varphi-\varphi+\varphi- Q_0\varphi, \beta\nabla u \cdot \bn - Q_b(\beta\nabla u \cdot \bn)  \rangle_{\partial T} \right|\\
		\leqslant & C\beta_2 \left( \sum_{T \in \mathcal{T}_h^I} 
		\|Q_0 \varphi- \varphi  \|_{\partial T}^2+\|Q_b \varphi- \varphi  \|_{\partial T}^2 \right)^{\frac{1}{2}} \left( \sum_{T \in \mathcal{T}_h^I} 
		\|\nabla u \cdot \bn - Q_b(\nabla u \cdot \bn) \|_{\partial T}^2 \right)^{\frac{1}{2}}\nonumber\\
		\leqslant&C\beta_2 h^{k+1}(\|u_1\|_{k+1,\Omega_1}+\|u_2\|_{k+1,\Omega_2})\|\varphi\|_2.\nonumber
	\end{align}
	(3) For  $\ell_2(u, \widetilde{Q}_h \varphi)$, $\widetilde{\ell}_3(u, \widetilde{Q}_h \varphi)$ and $\ell_6(u, \widetilde{Q}_h \varphi)$, we have
	\begin{eqnarray}\label{proof_L2_8}
		\begin{split}
			\ell_2(u, \widetilde{Q}_h \varphi) =& \sum_{T \in \mathcal{T}_h^I} \left( \nabla(Q_0 u -u), \beta\nabla_w \widetilde{Q}_h \varphi \right)_T\\
			=&\sum_{T \in \mathcal{T}_h^I} (\nabla(Q_0 u -u), \beta\nabla_w \widetilde{Q}_h\varphi - \beta\mathcal{Q}_h (\nabla \varphi) )_T\\
			&+\sum_{T \in \mathcal{T}_h^I} (\nabla(Q_0 u -u),  \beta\mathcal{Q}_h \nabla \varphi -\beta \nabla \varphi )_T\\ 
			&+\sum_{T \in \mathcal{T}_h^I} (\nabla(Q_0 u -u),  \beta \nabla \varphi )_T\\
			=&\ell_{21}(u, \varphi)+\ell_{22}(u, \varphi)+\ell_{23}(u, \varphi),
		\end{split}	
	\end{eqnarray}
	and
	\begin{align}\label{proof_L2_9}
		\widetilde{\ell}_3(u, \widetilde{Q}_h \varphi)=&\sum_{T \in \mathcal{T}_h^I} \langle Q_b(Q_0 u)- Q_{\delta} u, \beta \nabla_w \widetilde{Q}_h \varphi \cdot \bn \rangle_{\partial T}\nonumber\\
		=& \sum_{T \in \mathcal{T}_h^I} \langle
		Q_b(Q_0 u)- Q_{\delta} u, \beta \nabla_w \widetilde{Q}_h \varphi \cdot \bn - \beta \mathcal{Q}_h \nabla \varphi \cdot \bn  \rangle_{\partial T}\nonumber\\
		&+\sum_{T \in \mathcal{T}_h^I} \langle 	Q_b(Q_0 u)- Q_{\delta} u,  \beta \mathcal{Q}_h \nabla \varphi \cdot \bn - \beta \nabla \varphi \cdot \bn \rangle_{\partial T}\\
		&+\sum_{T \in \mathcal{T}_h^I} \langle Q_b(Q_0 u)- Q_{\delta} u, \beta \nabla \varphi \cdot \bn -\beta Q_b(\nabla \varphi \cdot \bn) \rangle_{\partial T}\nonumber\\
		&+\sum_{T \in \mathcal{T}_h^I} \langle Q_b(Q_0 u)- Q_{\delta} u, \beta Q_b(\nabla \varphi \cdot \bn) \rangle_{\partial T}\nonumber\\
		=&\ell_{31}(u, \varphi)+\ell_{32}(u, \varphi)+\ell_{33}(u, \varphi)+\ell_{34}(u, \varphi),\nonumber
	\end{align}
	and
	\begin{eqnarray}
		\begin{split}
			&\ell_6(u, \widetilde{Q}_h \varphi)\\
			=&\sum_{T \in \mathcal{T}_h^n}(\beta \nabla(\Pi_h u-u),\nabla(\Pi_h \varphi))_T \\
			=&\sum_{T \in \mathcal{T}_h^n}(\beta \nabla(\Pi_h u-u),\nabla(\Pi_h \varphi-\varphi))_T+\sum_{T \in \mathcal{T}_h^n}(\beta \nabla(\Pi_h u-u),\nabla\varphi)_T \\
			=&\ell_{61}(u, \varphi)+\ell_{62}(u, \varphi).
		\end{split}
	\end{eqnarray}
	Denote by $\theta(u,\varphi) = \ell_{23}(u, \varphi) - \ell_{34}(u, \varphi)+\ell_{62}(u, \varphi)$, then we have
	\begin{eqnarray*}
		\sum_{T \in \mathcal{T}_h^n}\langle \Pi_h u -u, \beta \nabla \varphi \cdot \bn \rangle_{\partial T } =-\sum_{e \in \mathcal{E}_h^n}\langle \Pi_h u -u, \beta \nabla \varphi \cdot \bn_e \rangle_e.
	\end{eqnarray*}
	Therefore, using integration by parts and the definition of $L^2$ projection operator $Q_b$ leads to 
	\begin{align*}
		\theta(u,\varphi)=& \sum_{T \in \mathcal{T}_h^I} -(Q_0 u -u, \nabla \cdot (\beta \nabla \varphi))_T
		+ \sum_{T \in \mathcal{T}_h^I} \langle Q_0 u -u, \beta \nabla \varphi \cdot \bn  \rangle_{\partial T}\\
		&-\sum_{T \in \mathcal{T}_h^I} \langle Q_0 u- u, \beta Q_b(\nabla \varphi \cdot \bn) \rangle_{\partial T \cap \mathcal{E}_h^I}\\
		&-\sum_{T \in \mathcal{T}_h^I} \langle Q_0 u- \Pi_h u, \beta Q_b(\nabla \varphi \cdot \bn) \rangle_{\partial T \cap \mathcal{E}_h^n}\\
		&-\sum_{T \in \mathcal{T}_h^n}(\Pi_h u -u,\beta \nabla \cdot ( \nabla \varphi))_T-\sum_{e \in \mathcal{E}_h^n} \langle \Pi_h u -u, \beta \nabla \varphi \cdot \bn_e  \rangle_e,\\
		=& \sum_{T \in \mathcal{T}_h^I} -(Q_0 u -u, \nabla \cdot (\beta \nabla \varphi))_T
		+ \sum_{T \in \mathcal{T}_h^I} \langle Q_0 u -u, \beta \nabla \varphi \cdot \bn - \beta Q_b(\nabla \varphi \cdot \bn) \rangle_{\partial T }\\
		&-\sum_{T \in \mathcal{T}_h^n}(\Pi_h u -u,\beta \nabla \cdot ( \nabla \varphi))_T+\sum_{e \in \mathcal{E}_h^n} \langle \Pi_h u -u,  \beta Q_b(\nabla \varphi \cdot \bn_e)-\beta \nabla \varphi \cdot \bn_e   \rangle_e.
	\end{align*}
	Therefore, we get
	\begin{eqnarray}\label{proof_L2_11}
		\begin{split}
			&\ell_2(u,\widetilde{Q}_h \varphi)-	\ell_3(u,\widetilde{Q}_h \varphi)+\ell_6(u, \widetilde{Q}_h \varphi)\\
			=&\ell_{21}(u, \varphi)+\ell_{22}(u, \varphi)-\ell_{31}(u, \varphi)-\ell_{32}(u, \varphi) -\ell_{33}(u, \varphi)+\ell_{61}(u, \varphi) +\theta(u, \varphi).
		\end{split}
	\end{eqnarray}
	Now, we estimate each term on the right side of the above equation. 
	
	(i) For $\ell_{21}(u, \varphi)$, according to the Cauchy-Schwarz inequality and the projection inequality \eqref{projectorestimate1}, we have
	\begin{eqnarray}\label{proof_L2_12}
		\begin{split}
			|\ell_{21}(u, \varphi)|
			=&\left| \sum_{T \in \mathcal{T}_h^I}(\nabla(Q_0 u -u), \beta\nabla_w \widetilde{Q}_h \varphi - \beta\mathcal{Q}_h \nabla \varphi)_T  \right|\\
			\leqslant & C\beta_2 \left(\sum_{T \in \mathcal{T}_h^I}\|\nabla(Q_0 u -u)\|_T^2 \right)^{\frac{1}{2}} \left( \sum_{T \in \mathcal{T}_h^I} \| \nabla_w \widetilde{Q}_h \varphi - \mathcal{Q}_h \nabla \varphi \|_T^2\right)^{\frac{1}{2}}\\
			\leqslant & C\beta_2 h^k(\|u_1\|_{k+1,\Omega_1}+\|u_2\|_{k+1,\Omega_2}) \left( \sum_{T \in \mathcal{T}_h^I} \| \nabla_w \widetilde{Q}_h \varphi - \mathcal{Q}_h \nabla \varphi \|_T^2\right)^{\frac{1}{2}}.
		\end{split}
	\end{eqnarray}
	For any ${\bf{q}} \in \nabla \mathcal{V}_k(T)$, it follows from Eq.(\ref{weak_gradient_exchange_2}), the Cauchy-Schwarz inequality, the trace inequality, Lemma \ref{inverse_interface_element}, the projection inequality \eqref{projectorestimate1} and the interpolation estimate that
	\begin{eqnarray}\label{proof_L2_26}
		\begin{split}
			&\left|( \nabla_w \widetilde{Q}_h \varphi -  \mathcal{Q}_h \nabla \varphi, \beta {\bf{q}})_T\right|\\
			=&\Big|(\nabla(Q_0 \varphi - \varphi), \beta{\bf{q}})_T- \langle Q_b(Q_0 \varphi)- Q_{\delta}\varphi, \beta {\bf{q}} \cdot \bn \rangle_{\partial T \cap \mathcal{E}_h^I }\Big| \\
			\leqslant &C \beta_2\|\nabla(Q_0 \varphi - \varphi)\|_T \|{\bf{q}}\|_T +C \|Q_0 \varphi-\varphi  \|_{\partial T} \| \beta {\bf{q}} \cdot \bn\|_{\partial T }\\
			&+C \|\Pi_h \varphi-\varphi  \|_{\partial T \cap \mathcal{E}_h^n} \| \beta {\bf{q}} \cdot \bn\|_{\partial T \cap \mathcal{E}_h^n }\\
			\leqslant & C\beta_2 h\|\varphi\|_2\|{\bf{q}}\|_T.
		\end{split}
	\end{eqnarray}
	Taking ${\bf{q}}=  \nabla_w \widetilde{Q}_h \varphi -  \mathcal{Q}_h \nabla \varphi $ in Eq.(\ref{proof_L2_26}) leads to 
	\begin{eqnarray}\label{proof_L2_14}
		\|\nabla_w \widetilde{Q}_h \varphi -  \mathcal{Q}_h \nabla \varphi\|_T \leqslant C \frac{\beta_2}{\beta_1} h\|\varphi\|_2.
	\end{eqnarray}
	Substituting Eq.(\ref{proof_L2_14}) into Eq.(\ref{proof_L2_12}) yields 
	\begin{eqnarray}\label{proof_L2_15}
		|\ell_{21}(u,\varphi)|\leqslant C\frac{\beta_2^2}{\beta_1} h^{k+1}(\|u_1\|_{k+1,\Omega_1}+\|u_2\|_{k+1,\Omega_2})\|\varphi\|_2.
	\end{eqnarray}
	
	(ii) For $\ell_{22}(u, \varphi)$, by the Cauchy-Schwarz inequality and the projection inequalities \eqref{projectorestimate1}-\eqref{projectorestimate2}, we have
	\begin{eqnarray}\label{proof_L2_16}
		\begin{split}
			|\ell_{22}(u, \varphi)|
			=&\left|\sum_{T \in \mathcal{T}_h^I} (\nabla(Q_0 u -u ), \beta \mathcal{Q}_h \nabla \varphi - \beta\nabla \varphi)_T  \right|\\
			\leqslant & C \beta_2 \left(\sum_{T \in \mathcal{T}_h^I} \|\nabla(Q_0 u -u )\|_T^2 \right)^{\frac{1}{2}}
			\left(\sum_{T \in \mathcal{T}_h^I} \| \mathcal{Q}_h \nabla \varphi - \nabla \varphi\|_T^2 \right)^{\frac{1}{2}}\\
			\leqslant & C\beta_2 h^{k+1}(\|u_1\|_{k+1,\Omega_1}+\|u_2\|_{k+1,\Omega_2})\|\varphi\|_2.
		\end{split}
	\end{eqnarray}
	
	(iii) For $\ell_{31}(u, \varphi)$, according to the Cauchy-Schwarz inequality, the trace inequality, the {projection inequality \eqref{projectorestimate1}}, Lemma \ref{inverse_interface_element} and the estimate \eqref{proof_L2_14}, we obtain
	\begin{eqnarray}\label{proof_L2_17}
		\begin{split}
			&|\ell_{31}(u, \varphi)|\\
			=&\left|\sum_{T \in \mathcal{T}_h^I} 
			\langle Q_b(Q_0 u) - Q_{\delta} u, \beta \nabla_w \widetilde{Q}_h \varphi \cdot \bn - \beta \mathcal{Q}_h \nabla \varphi \cdot \bn
			\rangle_{\partial T} \right|\\
			\leqslant & C  \left(\sum_{T \in \mathcal{T}_h^I} \| Q_0 u - u\|_{\partial T}^2 +\sum_{T \in \mathcal{T}_h^I}\|\Pi_h u - u \|_{\partial T \cap \mathcal{E}_h^n}^2 \right)^{\frac{1}{2}}\\ 
			&\left(\sum_{T \in \mathcal{T}_h^I} \|  \beta \nabla_w \widetilde{Q}_h \varphi \cdot \bn - \beta \mathcal{Q}_h \nabla \varphi \cdot \bn  \|_{\partial T}^2 \right)^{\frac{1}{2}}\\
			\leqslant & C\frac{\beta_2^2}{\beta_1} h^{k+1}(\|u_1\|_{k+1,\Omega_1}+\|u_2\|_{k+1,\Omega_2})\|\varphi\|_2.
		\end{split}
	\end{eqnarray}
	Similarly, we get
	\begin{eqnarray}\label{proof_L2_18}
		\begin{split}
			&|\ell_{32}(u, \varphi)|\nonumber\\
			=&\left|\sum_{T \in \mathcal{T}_h^I}\langle Q_b(Q_0 u)- Q_{\delta} u,
			\beta \mathcal{Q}_h \nabla \varphi \cdot \bn - \beta \nabla \varphi \cdot \bn \rangle_{\partial T}  \right|\nonumber\\
			\leqslant & C \beta_2 \left( \sum_{T \in \mathcal{T}_h^I} \|Q_0 u -u \|_{\partial T}^2+\sum_{T \in \mathcal{T}_h^I}\|\Pi_h u - u \|_{\partial T \cap \mathcal{E}_h^n}^2 \right)^{\frac{1}{2}}\left( \sum_{T \in \mathcal{T}_h^I} \| \mathcal{Q}_h \nabla \varphi -  \nabla \varphi  \|_{\partial T}^2\right)^{\frac{1}{2}} \nonumber\\
			\leqslant &  C \beta_2 h^{k+1}(\|u_1\|_{k+1,\Omega_1}+\|u_2\|_{k+1,\Omega_2})\|\varphi\|_2,\nonumber
		\end{split}
	\end{eqnarray}
	and
	\begin{align}\label{proof_L2_19}
		&|\ell_{33}(u, \varphi)|
		=\left|\sum_{T \in \mathcal{T}_h^I} \langle Q_b(Q_0 u)- Q_{\delta} u, \beta \nabla \varphi \cdot \bn -\beta Q_b(\nabla \varphi \cdot \bn) \rangle_{\partial T}\right|\nonumber\\
		\leqslant & C \left(\sum_{T \in \mathcal{T}_h^I} \|Q_0 u - u\|_{\partial T}^2+\sum_{T \in \mathcal{T}_h^I}\|\Pi_h u - u \|_{\partial T \cap \mathcal{E}_h^n}^2 \right)^{\frac{1}{2}}\left(\sum_{T \in \mathcal{T}_h^I} \|\beta \nabla \varphi \cdot \bn -\beta Q_b(\nabla \varphi \cdot \bn)\|_{\partial T}^2 \right)^{\frac{1}{2}}\\
		\leqslant & C \beta_2 h^{k+1}(\|u_1\|_{k+1,\Omega_1}+\|u_2\|_{k+1,\Omega_2})\|\varphi\|_2,\nonumber
	\end{align}
	and
	\begin{align}\label{proof_L2_22}
		|\ell_{61}(u,\varphi)|
		=&\left| \sum_{T \in \mathcal{T}_h^n}(\beta \nabla(\Pi_h u-u),\nabla(\Pi_h \varphi-\varphi))_T   \right|\nonumber\\
		\leqslant & C \beta_2 \sum_{T \in \mathcal{T}_h^n} \|\nabla(\Pi_h u-u)\|_T \|\nabla(\Pi_h \varphi-\varphi)\|_T\\
		\leqslant& C \beta_2 \left(  \sum_{T \in \mathcal{T}_h^n} \|\nabla(\Pi_h u-u)\|_T^2  \right)^{\frac{1}{2}}\left(  \sum_{T \in \mathcal{T}_h^n} \|\nabla(\Pi_h \varphi-\varphi)\|_T^2  \right)^{\frac{1}{2}}\nonumber\\
		\leqslant&C \beta_2 h^{k+1}(\|u_1\|_{k+1,\Omega_1}+\|u_2\|_{k+1,\Omega_2})\|\varphi\|_2.\nonumber
	\end{align}
	For $\theta(u,\varphi)$, we obtain
	\begin{align}\label{proof_L2_20}
		&|\theta(u,\varphi)|\nonumber\\
		\leqslant&\left|\sum_{T \in \mathcal{T}_h^I} -(Q_0 u -u, \nabla \cdot (\beta \nabla \varphi))_T \right| +\left| \sum_{T \in \mathcal{T}_h^I} \langle Q_0 u -u, \beta \nabla \varphi \cdot \bn - \beta Q_b(\nabla \varphi \cdot \bn) \rangle_{\partial T}\right|\nonumber\\
		&+\left|\sum_{T \in \mathcal{T}_h^n}(\Pi_h u -u,\beta \nabla \cdot ( \nabla \varphi))_T \right|+\left|\sum_{e \in \mathcal{E}_h^n} \langle \Pi_h u -u,  \beta Q_b(\nabla \varphi \cdot \bn_e)-\beta \nabla \varphi \cdot \bn_e   \rangle_e \right|\nonumber\\
		\leqslant& C\left(\sum_{T \in \mathcal{T}_h^I} \|Q_0 u -u\|_T^2 \right)^{\frac{1}{2}}\left(\sum_{T \in \mathcal{T}_h^I} \|\nabla \cdot (\beta \nabla \varphi)\|_{T}^2 \right)^{\frac{1}{2}}\\
		&+C\left(\sum_{T \in \mathcal{T}_h^I} \|Q_0 u -u\|_{\partial T}^2 \right)^{\frac{1}{2}}\left(\sum_{T \in \mathcal{T}_h^I} \|\beta \nabla \varphi \cdot \bn - \beta Q_b(\nabla \varphi \cdot \bn)\|_{\partial T}^2 \right)^{\frac{1}{2}}\nonumber\\
		&+C\left(\sum_{T \in \mathcal{T}_h^n} \|\Pi_h u -u\|_T^2 \right)^{\frac{1}{2}}\left(\sum_{T \in \mathcal{T}_h^n} \|\beta\nabla \cdot ( \nabla \varphi)\|_{T}^2 \right)^{\frac{1}{2}}\nonumber\\
		&+C\left(\sum_{e \in \mathcal{E}_h^n} \|\Pi_h u -u\|_{\partial T}^2 \right)^{\frac{1}{2}}\left(\sum_{e \in \mathcal{E}_h^n} \|\beta \nabla \varphi \cdot \bn - \beta Q_b(\nabla \varphi \cdot \bn)\|_{\partial T}^2 \right)^{\frac{1}{2}}\nonumber\\
		\leqslant & C\beta_2 h^{k+1}(\|u_1\|_{k+1,\Omega_1}+\|u_2\|_{k+1,\Omega_2})\|\varphi\|_2.\nonumber
	\end{align}
	(5)For $\ell_5(u,\widetilde{Q}_h \varphi)$, we use the fact that $\sum_{e \in \mathcal{E}_h^n} \langle Q_b(\beta \nabla u \cdot \bn_e), \Pi_h \varphi - Q_b(\Pi_h \varphi) \rangle_e =0$ and the definition of $L^2$ projection operator $Q_b$ to get
	\begin{align}\label{proof_L2_28}
		&|\ell_5(u,\widetilde{Q}_h \varphi)|\nonumber\\
		=&\left| \sum_{e \in \mathcal{E}_h^n} \langle \beta\nabla u \cdot \bn_e, \Pi_h \varphi - Q_b (\Pi_h \varphi)\rangle_e\right| \nonumber\\
		=&\left|\sum_{e \in \mathcal{E}_h^n} \langle \beta\nabla u \cdot \bn_e - Q_b(\beta\nabla u \cdot \bn_e), \Pi_h \varphi - Q_b (\Pi_h \varphi)\rangle_e\right|\\
		\leqslant &C \left(\sum_{e \in \mathcal{E}_h^n} \|\beta\nabla u \cdot \bn_e - Q_b(\beta\nabla u \cdot \bn_e) \|_e^2 \right)^{\frac{1}{2}}\nonumber\\
		&\left(\sum_{e \in \mathcal{E}_h^n} \|\Pi_h \varphi - \varphi \|_e^2
		+\| \varphi - Q_b \varphi \|_e^2
		+\|Q_b \varphi - Q_b (\Pi_h \varphi) \|_e^2
		\right)^{\frac{1}{2}}\nonumber\\
		\leqslant & C\beta_2 h^{k+1}(\|u_1\|_{k+1,\Omega_1}+\|u_2\|_{k+1,\Omega_2})\| \varphi \|_{2}.\nonumber
	\end{align}
	Combining estimates (\ref{proof_L2_5})-(\ref{proof_L2_28}) and \eqref{proof_L2_25}-\eqref{proof_L2_4} yields 
	\begin{eqnarray}\label{proof_L2_21}
		\begin{split}
			\|\varepsilon_0\|^2\leqslant& C\frac{\beta_2^2}{\beta_1} h\|\varphi\|_{2}\|\varepsilon_h\|_{1,h}+C\frac{\beta_2^2}{\beta_1} h^{k+1}(\|u_1\|_{k+1,\Omega_1}+\|u_2\|_{k+1,\Omega_2})\|\varphi\|_{2}\\
			\leqslant& C{(\frac{\beta_2^4}{\beta_1^4} + \frac{\beta_2^2}{\beta_1^2})}  h^{k+1}(\|u_1\|_{k+1,\Omega_1}+\|u_2\|_{k+1,\Omega_2}) \|\varepsilon_0\|\\
			\leqslant & C {\frac{\beta_2^4}{\beta_1^4}} h^{k+1}(\|u_1\|_{k+1,\Omega_1}+\|u_2\|_{k+1,\Omega_2})\|\varepsilon_0\|.
		\end{split}
	\end{eqnarray}
	Therefore, we have
	\begin{align*}
		\|\varepsilon_0\|
		\leqslant C \left({\frac{\beta_2}{\beta_1}}\right)^4 h^{k+1}(\|u_1\|_{k+1,\Omega_1}+\|u_2\|_{k+1,\Omega_2}).
	\end{align*}
	The proof of the theorem is completed.
\end{proof}

\section{The Numerical Experiments}
In this section, we give two numerical experiments to demonstrate that the efficiency of the proposed immersed {\color{black}skeletal finite element} scheme. In these examples, we compute the following relative errors in the $H^1$ norm and $L^2$ norm.
\begin{eqnarray*}
	\|\bar{e}_h\|_{1,h}:= \frac{\|e_h\|_{1,h}}{\|P_hu\|_{1,h}},\quad \|\bar{e}_0\|:= \dfrac{\|e_0\|}{\|P_hu\|}.
\end{eqnarray*}
In two examples, we take $\beta_1=1$ and $\beta_2=\{1,10,100,1000\}$.

\begin{example}\label{li1}
	Consider the interface problem in the square domain $[-1,1]\times[-1,1]$. The interface function is as follows:
	$$
	x^2+y^2=\frac{1}{3}.
	$$
	Take two subdomains: $\Omega_1=\left\{(x,y)\in \mathbb{R}^2: x^2+y^2<\dfrac{1}{3}\right\}$ and  $\Omega_2=\left\{(x,y)\in \mathbb{R}^2: x^2+y^2>\dfrac{1}{3} \right\}$. The exact solution is
	$$
	u = \left\{\begin{array}{lc}
		\dfrac{1}{\beta_1}\cos(\pi (x^2+y^2)), & (x,y)\in \Omega_1, \\
		\dfrac{1}{\beta_2}\cos(\pi (x^2+y^2))+\dfrac{1}{2}(\dfrac{1}{\beta_1}-\dfrac{1}{\beta_2}), & (x,y)\in \Omega_2.
	\end{array}
	\right.
	$$
\end{example}

\begin{figure}[H]
	\centering
	\begin{minipage}[t]{0.33\linewidth}
		\centering
		\includegraphics[width=1.1\linewidth]{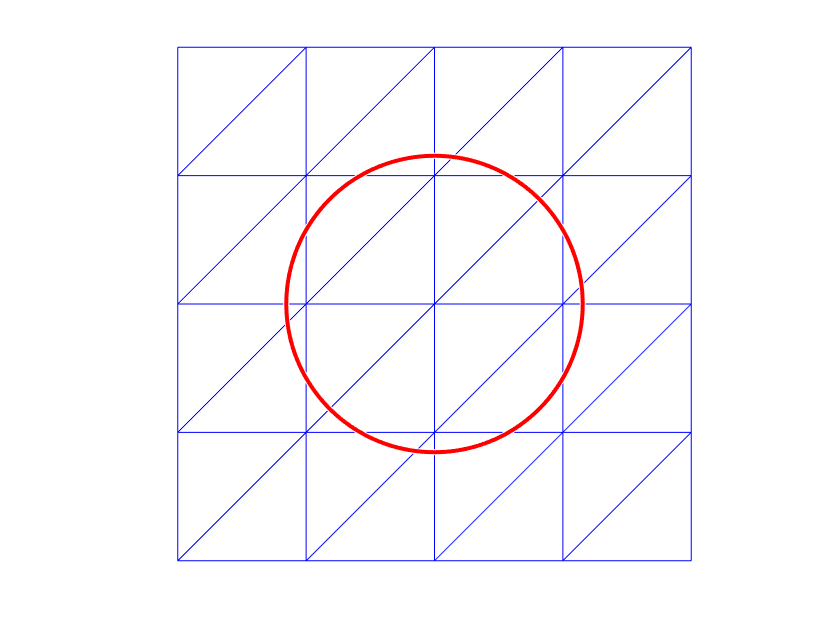}
	\end{minipage}%
	\begin{minipage}[t]{0.33\linewidth}
		\centering
		\includegraphics[width=1.1\linewidth]{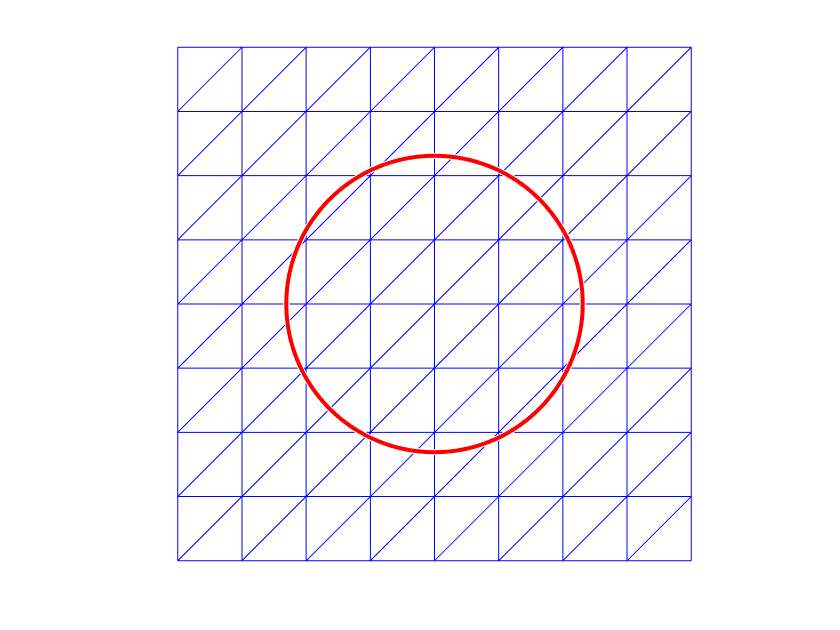}
	\end{minipage}
	\begin{minipage}[t]{0.33\linewidth}
		\centering \includegraphics[width=1.1\linewidth]{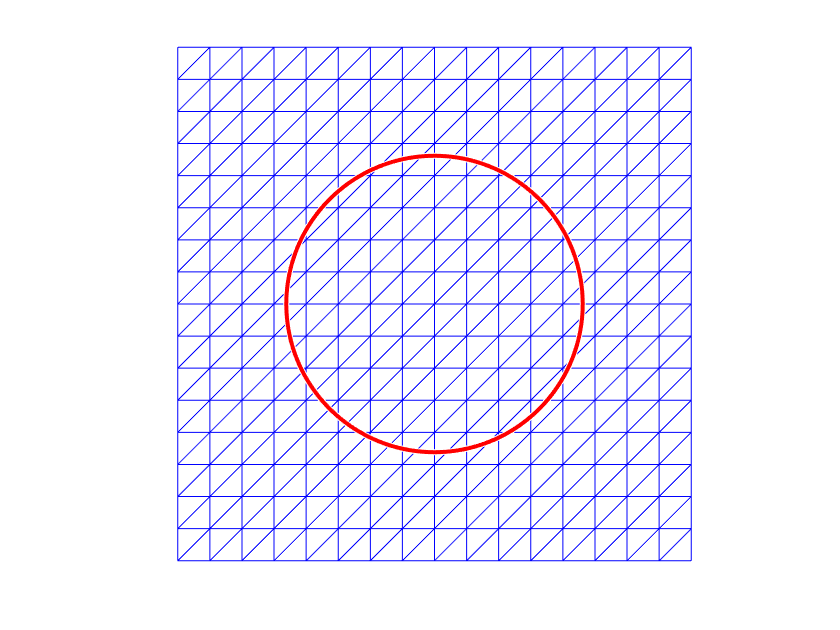}
	\end{minipage}
	\caption{ When $n$ = 4,8,16, the meshes of domain $\Omega$ in Example \ref{li1}.}
	\label{partition_2_li1_IEI}
	\vspace*{-0.5cm}
\end{figure}

	\begin{figure}[H]
		\centering
		\begin{minipage}[t]{0.49\linewidth}
			\centering
			\includegraphics[width=0.8\linewidth]{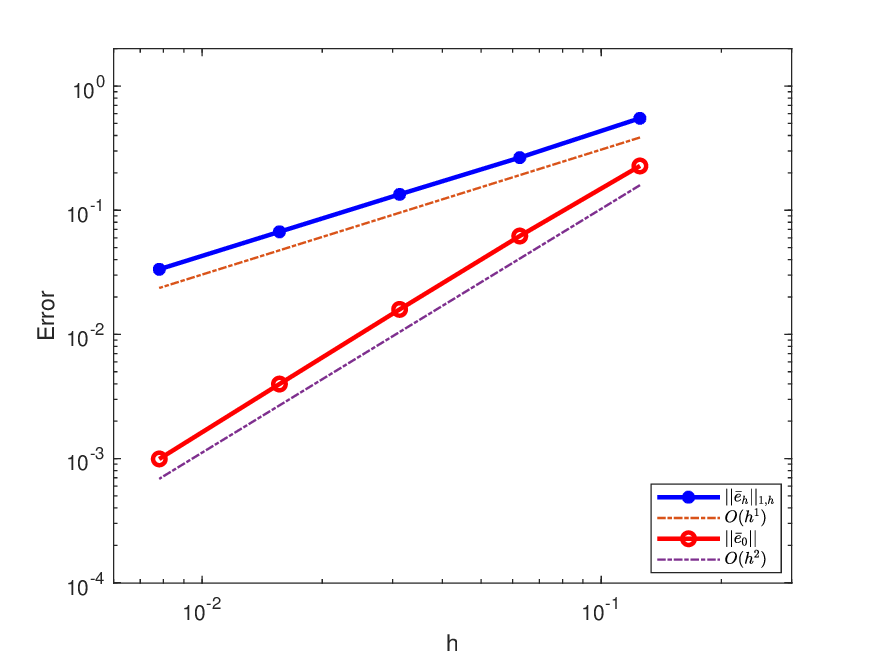}
		\end{minipage}
		\begin{minipage}[t]{0.49\linewidth}
			\centering
			\includegraphics[width=0.8\linewidth]{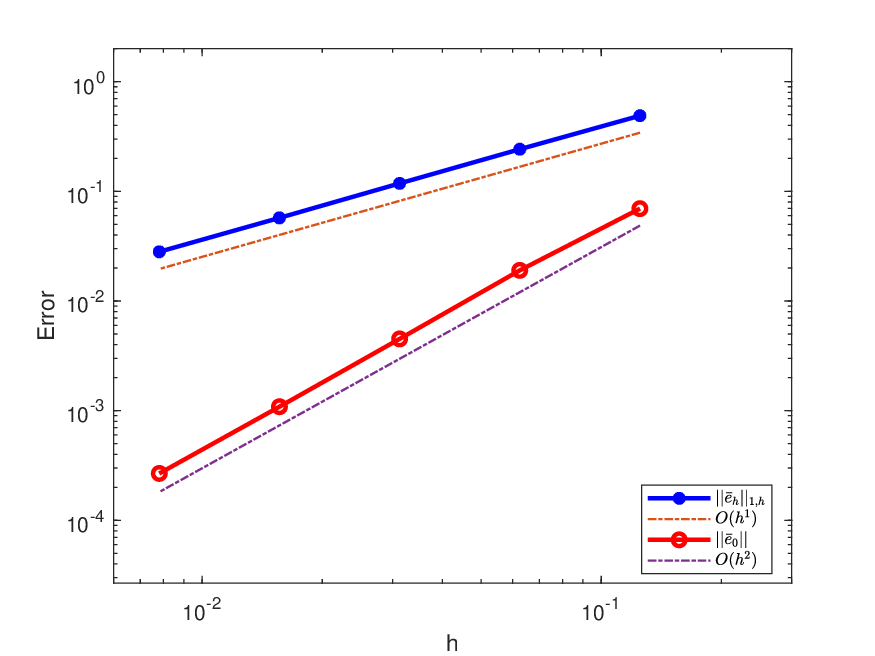}
		\end{minipage}\\
		\begin{minipage}[t]{0.49\linewidth}
			\centering \includegraphics[width=0.8\linewidth]{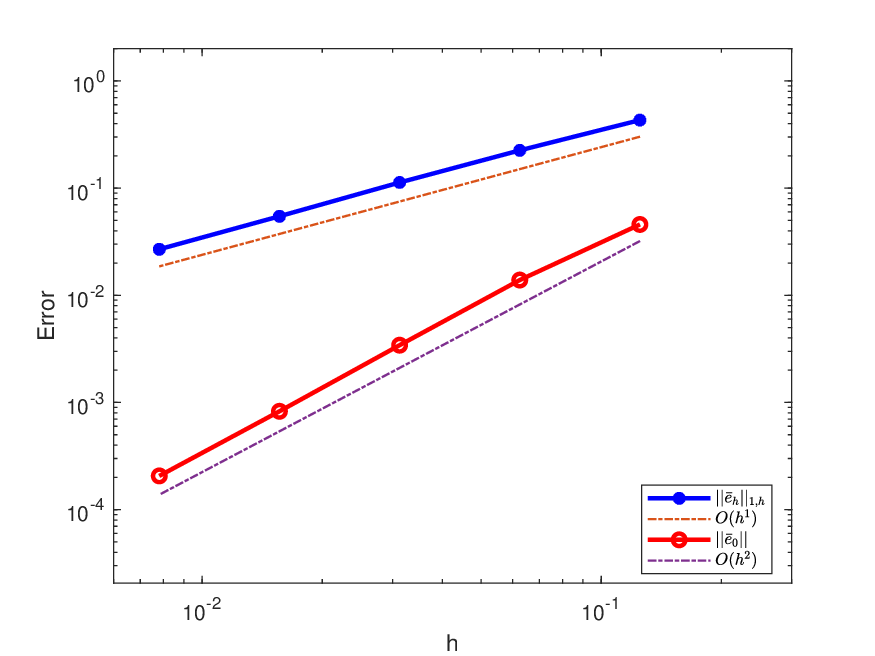}
		\end{minipage}
		\begin{minipage}[t]{0.49\linewidth}
			\centering \includegraphics[width=0.8\linewidth]{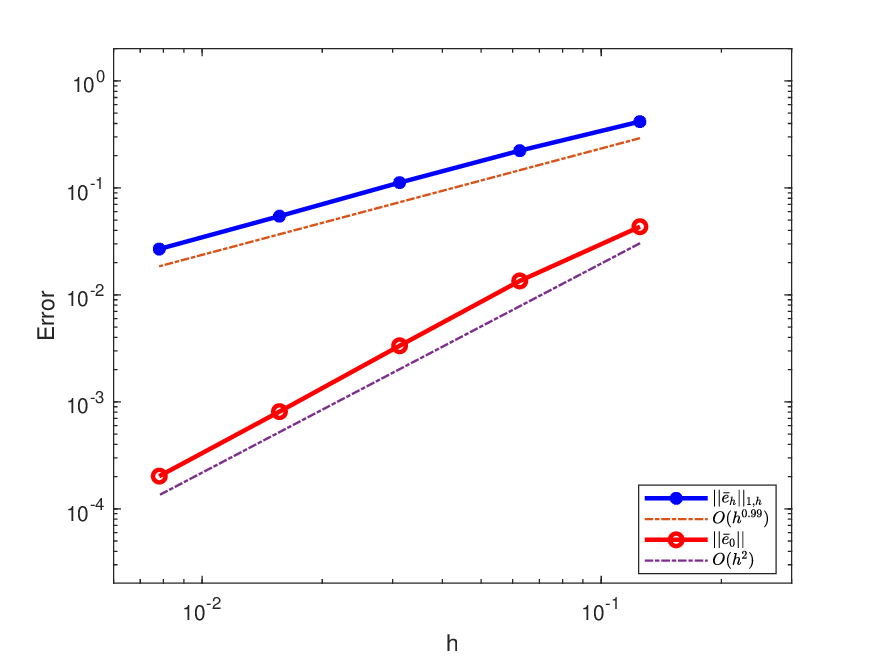}
		\end{minipage}
		\caption{The errors and convergence orders of the example \ref{li1} when we use $P_1$  {\color{black}skeletal finite element} (Upper left: $\beta_2=1$, Upper right: $\beta_2=10$, Lower left: $\beta_2=100$, Lower right: $\beta_2=1000$ ).}
		\label{Figure1_IEI}
		\vspace*{-0.3cm}
	\end{figure}
	
	We implement this example on triangular meshes (see Figure \ref{partition_2_li1_IEI}) and use $P_1$ to $P_2$ { {\color{black}skeletal finite} elements along with Lagrange elements to solve the second order elliptic interface problems.  The numerical results are shown in Figures \ref{Figure1_IEI} - \ref{Figure2_IEI}. From these figures, it's evident that the convergence orders are $O(h^k)$ in the $H^1$ norm and $O(h^{k+1})$ in the $L^2$ norm, respectively.
	
	\begin{figure}[H]
		\centering
		\begin{minipage}[t]{0.49\linewidth}
			\centering
			\includegraphics[width=0.8\linewidth]{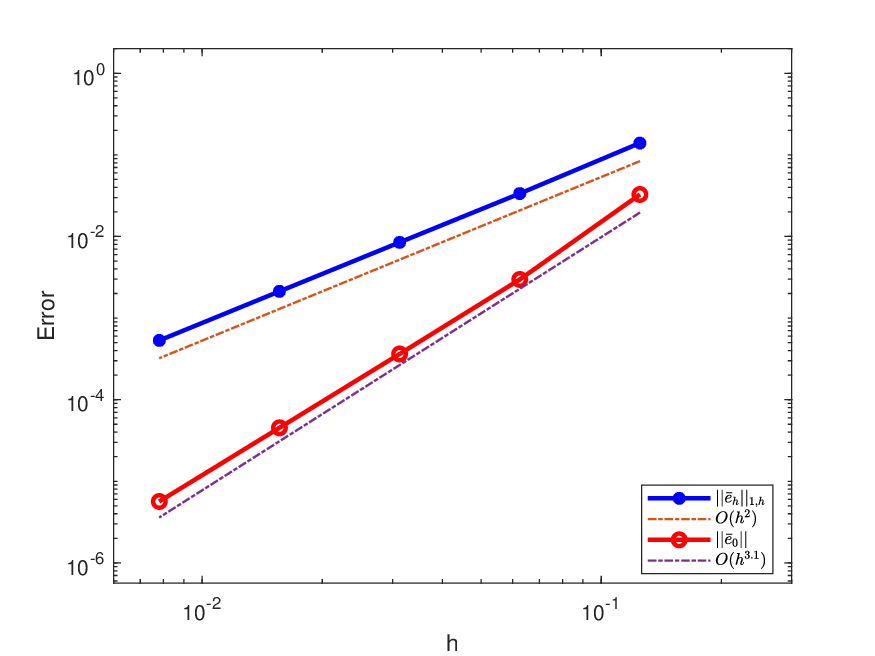}
		\end{minipage}%
		\begin{minipage}[t]{0.49\linewidth}
			\centering
			\includegraphics[width=0.8\linewidth]{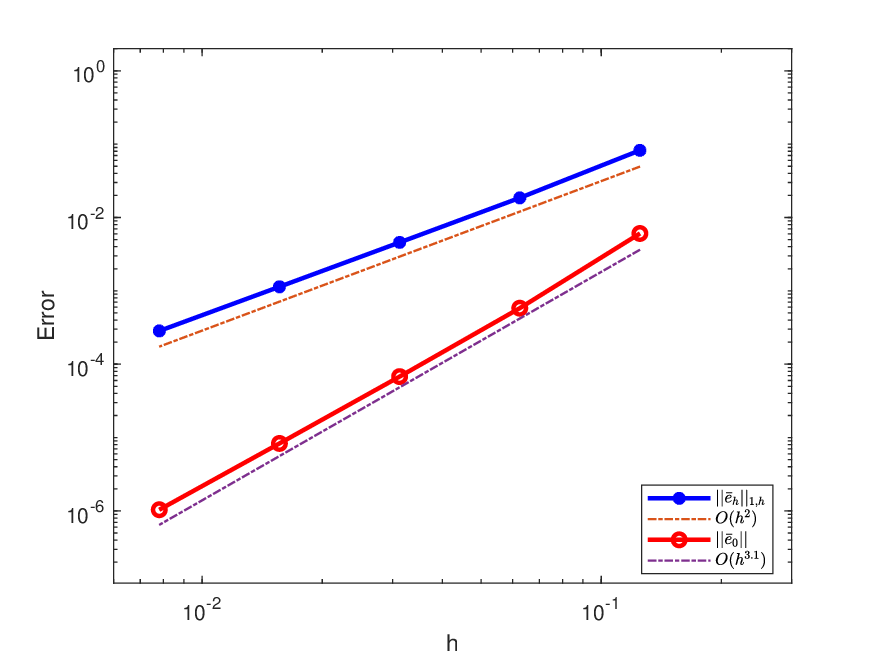}
		\end{minipage}\\
		\begin{minipage}[t]{0.49\linewidth}
			\centering \includegraphics[width=0.8\linewidth]{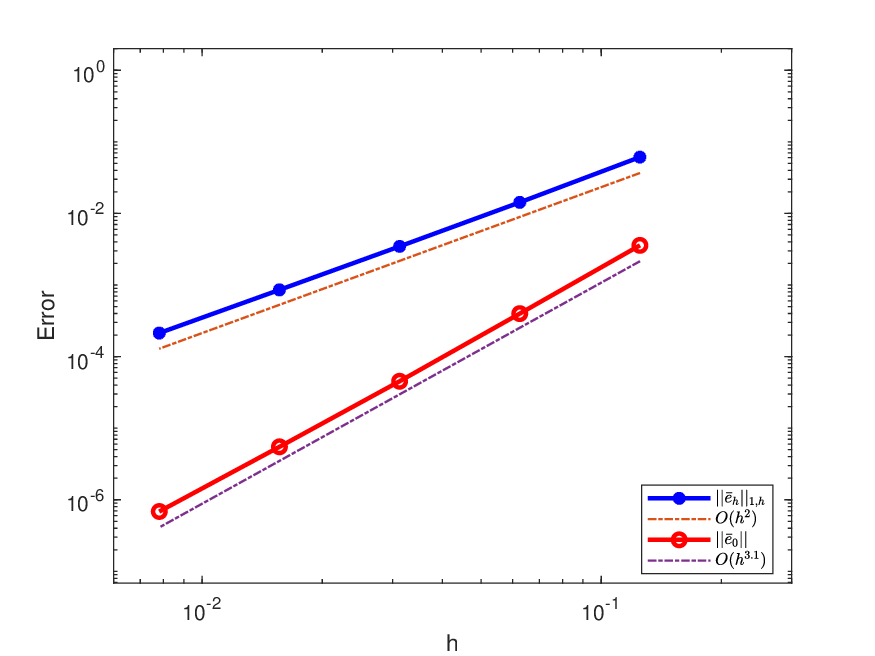}
		\end{minipage}
		\begin{minipage}[t]{0.49\linewidth}
			\centering \includegraphics[width=0.8\linewidth]{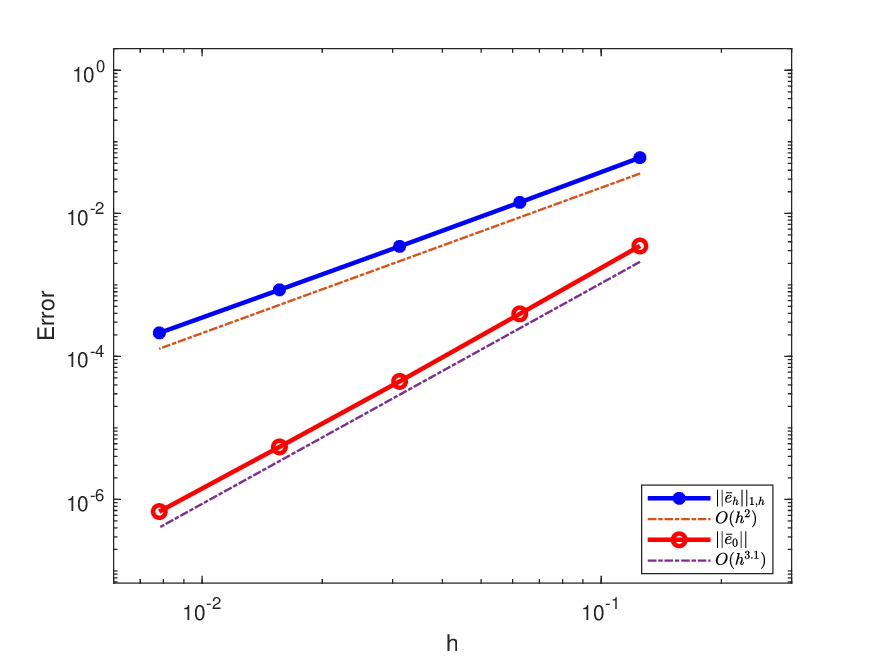}
		\end{minipage}
		\caption{The error and convergence orders of the example \ref{li1} when we use $P_2$ {\color{black}skeletal finite element} (Upper left: $\beta_2=1$, Upper right: $\beta_2=10$, Lower left: $\beta_2=100$, Lower right: $\beta_2=1000$ ).}
		\label{Figure2_IEI}
		\vspace*{-0.3cm}
	\end{figure}

	\begin{example}\label{li2}
		Consider the interface problem in the square domain $[0.6,1.6]\times[0.21,1.21]$. The interface is as follows:
		$$
		(x^2-y^2)^2-4x^2y^2+\frac{1}{2}=0.
		$$
		Let $\Omega_1=\left\{(x,y)\in \Omega| L(x,y)>0 \right\}$ and  $\Omega_2=\left\{(x,y)\in \Omega| L(x,y)<0 \right\}$. Next, we define 	$\widetilde{L}$  to be the harmonic conjugate of $L$ given by $\widetilde{L}=4xy(x^2-y^2)$. And we choose $f$ and $g$ such that 
		$$
		u(x,y)=\frac{1}{\beta}L(x,y)+\widetilde{L}(x,y)+\frac{1}{\beta}L(x,y)\widetilde{L}(x,y).
		$$
	\end{example}

	\begin{figure}[H]
		\centering
		\begin{minipage}[t]{0.33\linewidth}
			\centering
			\includegraphics[width=1.1\linewidth]{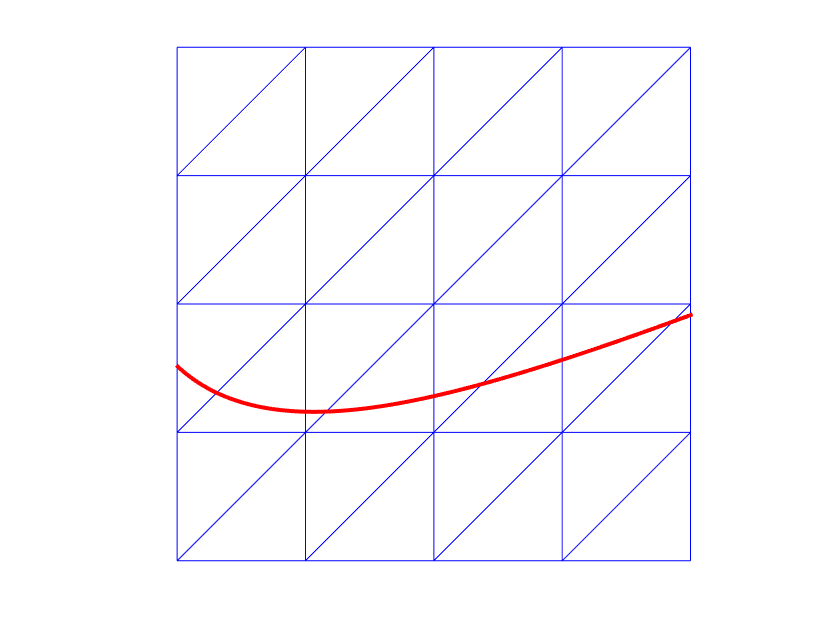}
		\end{minipage}%
		\begin{minipage}[t]{0.33\linewidth}
			\centering
			\includegraphics[width=1.1\linewidth]{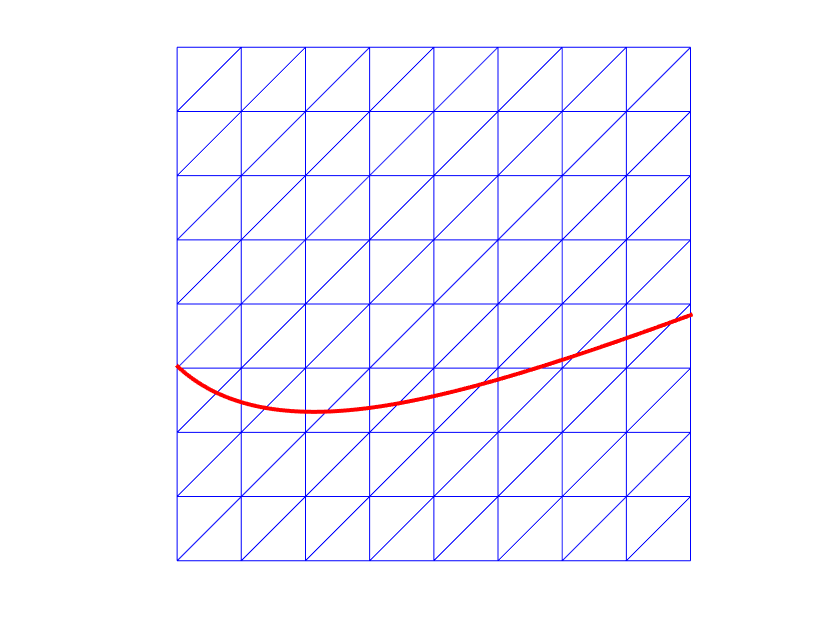}
		\end{minipage}
		\begin{minipage}[t]{0.33\linewidth}
			\centering \includegraphics[width=1.1\linewidth]{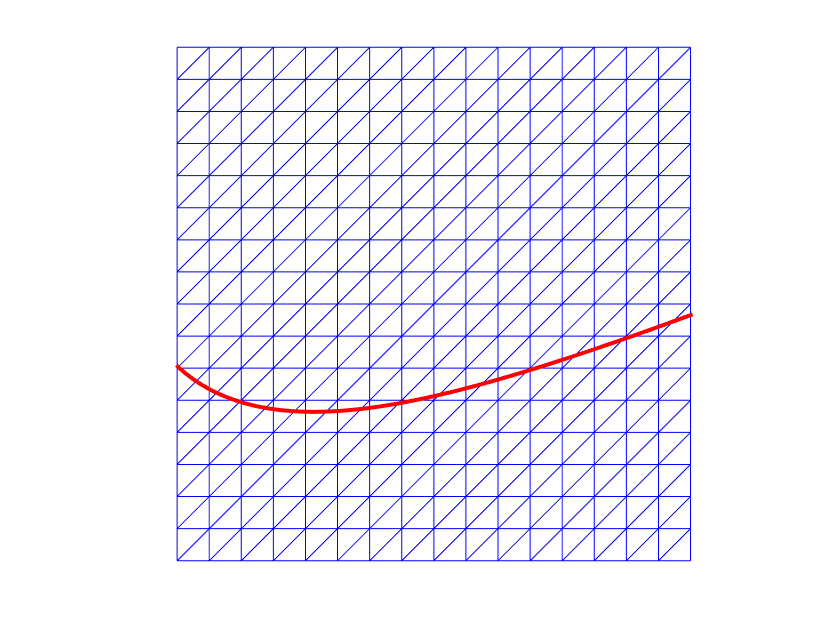}
		\end{minipage}
		\caption{ When $n$ = 4,8,16, the meshes of domain $\Omega$ in Example \ref{li2}.}
		\label{partition_2_li2_IEI}
	\end{figure}

	\begin{figure}[H]
		\centering
		\begin{minipage}[t]{0.49\linewidth}
			\centering
			\includegraphics[width=0.8\linewidth]{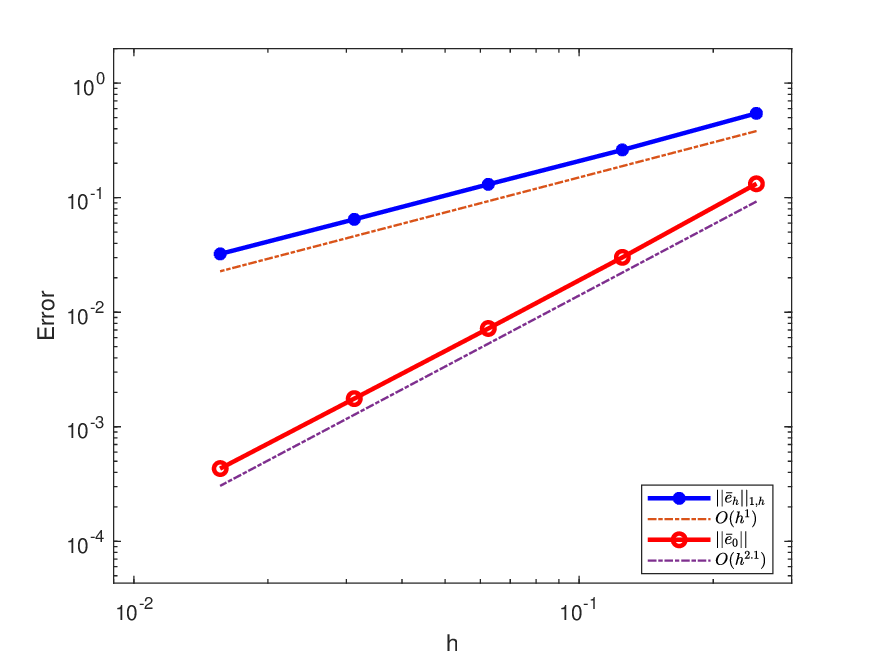}
		\end{minipage}%
		\begin{minipage}[t]{0.49\linewidth}
			\centering
			\includegraphics[width=0.8\linewidth]{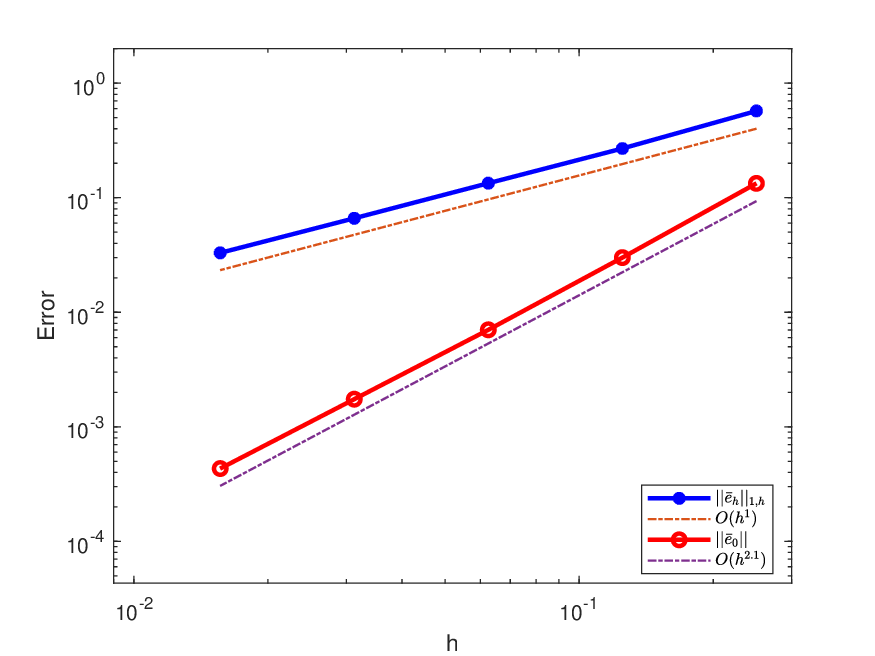}
		\end{minipage}\\
		\begin{minipage}[t]{0.49\linewidth}
			\centering \includegraphics[width=0.8\linewidth]{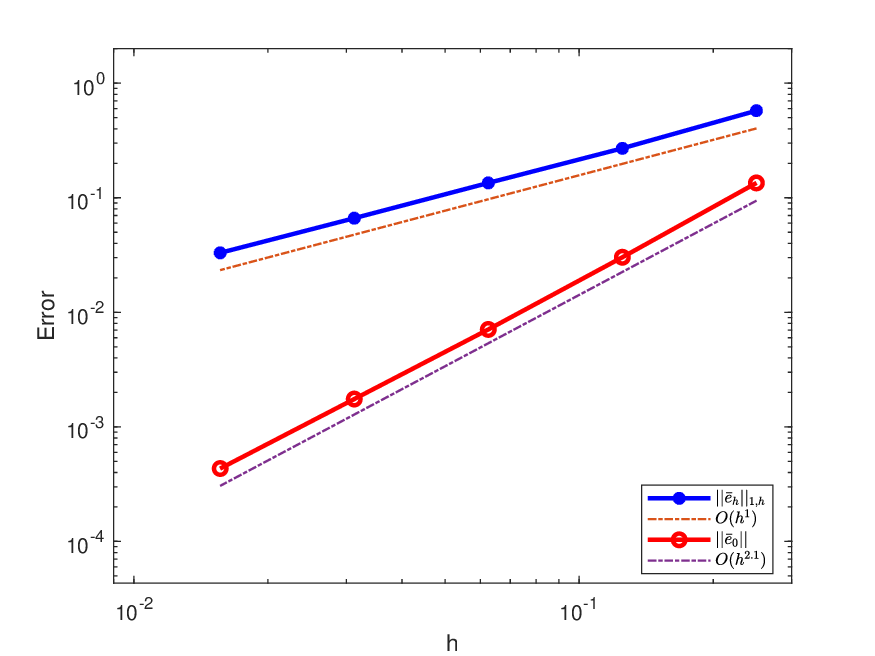}
		\end{minipage}
		\begin{minipage}[t]{0.49\linewidth}
			\centering \includegraphics[width=0.8\linewidth]{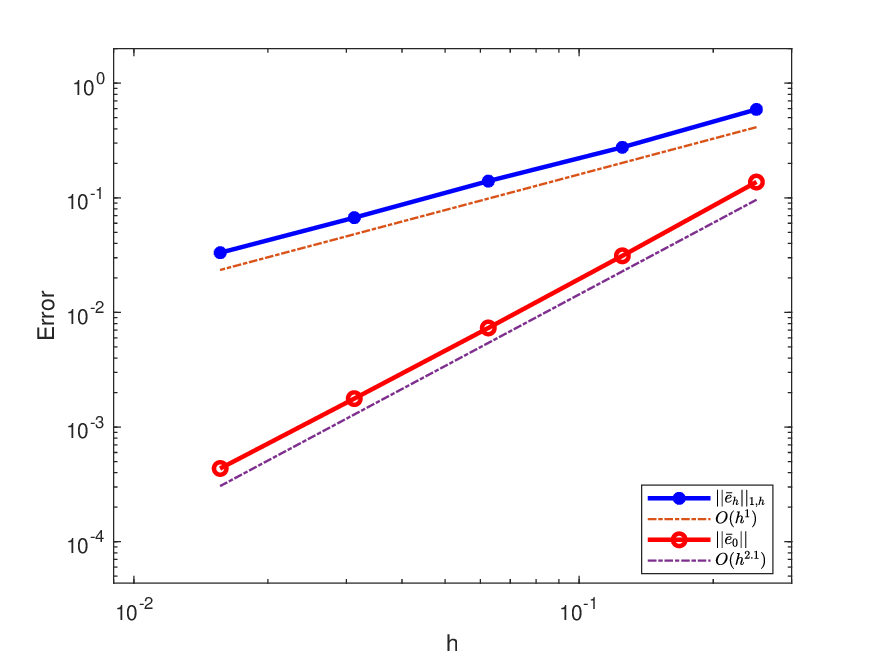}
		\end{minipage}
		\caption{The error and convergence orders of the example \ref{li2} when we use $P_1$ {\color{black}skeletal finite element}  (Upper left: $\beta_2=1$, Upper right: $\beta_2=10$, Lower left: $\beta_2=100$, Lower right: $\beta_2=1000$ ).}
		\label{Figure3_IEI}
		\vspace*{-0.3cm}
	\end{figure}
	
	Note that the tangential derivative of the exact solution $u$ and jump in the coefficient $\beta$ are non-zero on the interface $\Gamma$. Furthermore, it can be verified that $u$ satisfies the Laplacian extended jump conditions. The numerical results are shown in Figures \ref{Figure3_IEI} - \ref{Figure4_IEI} based on the triangular meshes (see Figure \ref{partition_2_li2_IEI} ). From these figures, we observe that the convergence orders remain optimal in both the $H^1$ norm and $L^2$ norm. These results demonstrate the effectiveness of the proposed numerical scheme for solving second order elliptic interface problems on unfitted meshes.
	
	\begin{figure}[H]
		\centering
		\begin{minipage}[t]{0.49\linewidth}
			\centering
			\includegraphics[width=0.8\linewidth]{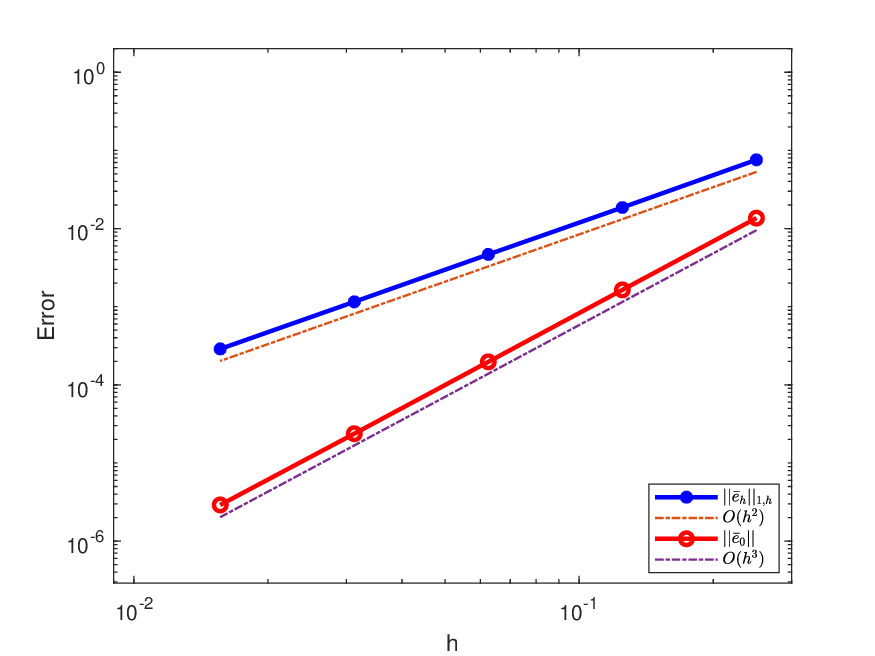}
		\end{minipage}%
		\begin{minipage}[t]{0.49\linewidth}
			\centering
			\includegraphics[width=0.8\linewidth]{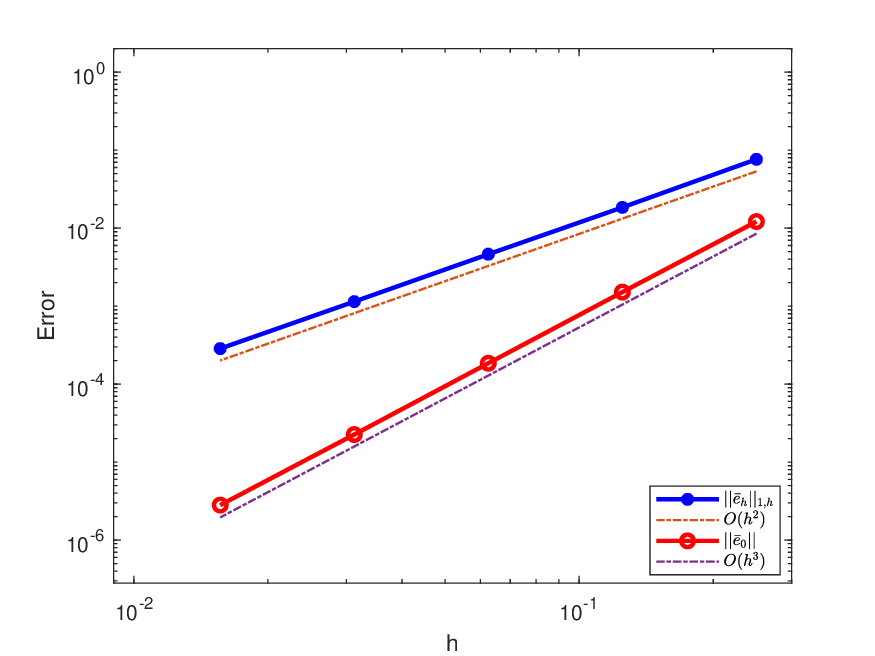}
		\end{minipage}\\
		\begin{minipage}[t]{0.49\linewidth}
			\centering \includegraphics[width=0.8\linewidth]{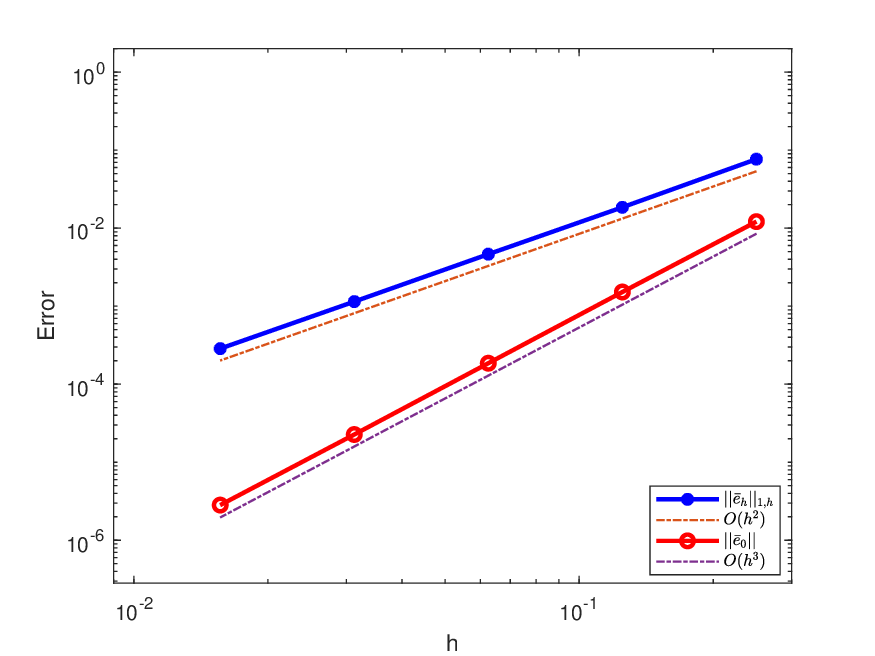}
		\end{minipage}
		\begin{minipage}[t]{0.49\linewidth}
			\centering \includegraphics[width=0.8\linewidth]{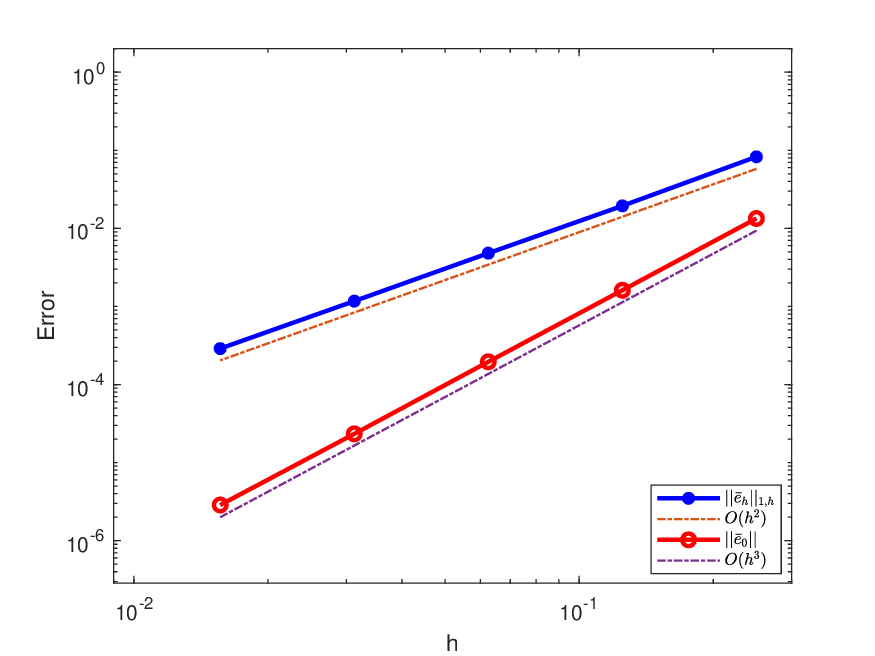}
		\end{minipage}
		\caption{The error and convergence orders of the example \ref{li2} when we use $P_2$ WG element (Upper left: $\beta_2=1$, Upper right: $\beta_2=10$, Lower left: $\beta_2=100$, Lower right: $\beta_2=1000$ ).}
		\label{Figure4_IEI}
	\end{figure}

	\section{Conclusion}In this paper, we use the {\color{black}skeletal finite element method}  and the standard finite element method to solve the second order elliptic interface problems on unfitted meshes. We employ IFE functions that precisely satisfy the jump conditions in the interior of interface elements. On the boundaries of these elements, we define a new boundary function space. A high order numerical scheme is proposed. And we prove that numerical solutions converge to the exact solutions at optimal rates. Additionally, numerical results from various examples demonstrate that the optimal convergence orders are achieved in both the $H^1$ norm and the $L^2$ norm. These results align with the theoretical analysis.

	\appendix
	\section{Some Important Inequalities}
	In this appendix, we present several important inequalities that are crucial for our proof. 
	\begin{lemma}[Trace inequality]\label{Trace_inequality}
		For the shape-regular partition $\mathcal{T}_h$, there exists a constant $C$ such that for any $T \in \mathcal{T}_h$ and $e \in \partial T$, the following inequality holds 
		\begin{eqnarray}\label{trace_inequality}
			\|w\|_e^2 \leqslant C\left(h_T^{-1}\|w\|_T^2 + h_T\|\nabla w\|_T^2 \right),
		\end{eqnarray}
		where $w \in H^1(T)$ is any function.
	\end{lemma}
	
	\begin{lemma}[Inverse inequality]
		For the shape-regular partition $\mathcal{T}_h$, there exists a constant $C$ such that for any $T \in \mathcal{T}_h$ and $e \in \partial T$, we have 
		\begin{eqnarray}\label{Inverse_inequality}
			\|\nabla \varphi\|_T \leqslant Ch_T^{-1}\|\varphi\|_T,
		\end{eqnarray}
		where $\varphi$ is any piecewise polynomial on $\mathcal{T}_h$.
	\end{lemma}

	{Referring to the proof of \cite[Theorem 2]{IFE_2025}, we conclude that the following estimates for triangular elements.}
	\begin{lemma}\cite{IFE_2025}\label{inverse_interface_element}
		For the interface element $T \in \mathcal{T}_h^I$ and $w \in \mathcal{V}_k(T)$, we have
		\begin{eqnarray}\label{inequality_1}
			\|\beta \nabla w \|_{\partial T}^2 \leqslant C\beta_2h_T^{-1}\|\nabla w\|_T^2.
		\end{eqnarray}
	\end{lemma}
	
	{By combining \cite[Theorem 2]{Frenet_Serret_2024} with \cite[Lemma 15]{IFE_2025}, we derive the following projection estimates.}
	\begin{lemma}\cite{Frenet_Serret_2024}
		For $\rho \in H^{l+1}(\Omega)$ with 
		$0 \leqslant l \leqslant k$ and $s=0,1$, we have the following estimates hold true
		\begin{align}
			\sum_{T \in \mathcal{T}_h^I} h_T^{2s} \| \rho- Q_0 \rho\|_{T,s}^2 \leqslant& Ch^{2(l+1)}\|\rho\|_{l+1}^2,\label{projectorestimate1}\\
			\sum_{T \in \mathcal{T}_h} h_T^{2s} \|\nabla \rho- \mathcal{Q}_h(\nabla \rho) \|_{T,s}^2 \leqslant& C h^{2l}\|\rho\|_{l+1}^2\label{projectorestimate2}.
		\end{align}
		Here $C$ represents a positive constant that remains independent of both the mesh size $h$, the relative location of the interface and the function involved in the above estimates.
	\end{lemma}
	
	\begin{lemma}
		For any $v_h=\{v_0, v_b\} \in V_h$, there exists a positive constant that independent of viscosity coefficient $\beta$ and interface location such that 
		\begin{eqnarray}\label{gradient_energy}
			\sum_{T \in \mathcal{T}_h^I} \|\nabla v_0\|_T^2 \leqslant C \left({\frac{\beta_2}{\beta_1}}\right)^2\| v_h\|_{1,h}^2.
		\end{eqnarray}
	\end{lemma}
	\begin{proof}
		For $T \in \mathcal{T}_h^I$ and ${\bf{q}} \in \nabla \mathcal{V}_k$, using Eq.(\ref{Def_weak_gradient_1}) leads to
		\begin{eqnarray}\label{gradient_energy_proof_1}
			(\beta \nabla_w v, {\bf{q}})_T=(\beta\nabla v_0, {\bf{q}})_T-\langle Q_b v_0 -v_b, \beta{\bf{q}} \cdot \bn \rangle_{\partial T}.
		\end{eqnarray}
		Taking ${\bf{q}} = \nabla v_0$ in Eq.\eqref{gradient_energy_proof_1} and applying the Cauchy-Schwarz inequality and Lemma \ref{inverse_interface_element} leads to
		\begin{align}\label{gradient_energy_proof_2}
			\beta_1\sum_{T \in \mathcal{T}_h^I} \|\nabla v_0\|_T^2\leqslant&	\sum_{T \in \mathcal{T}_h^I}(\beta \nabla v_0,\nabla v_0)_T\nonumber\\
			=&\sum_{T \in \mathcal{T}_h^I}(\beta\nabla_w v, \nabla v_0)_T+\sum_{T \in \mathcal{T}_h^I}\langle Q_b v_0 -v_b, \beta\nabla v_0 \cdot \bn \rangle_{\partial T}\nonumber\\
			\leqslant&C \sum_{T \in \mathcal{T}_h^I}  \left(\beta_2\|\nabla_w v\|_T\|\nabla v_0\|_T+\|Q_b v_0 -v_b\|_{\partial T}\|\beta \nabla v_0 \cdot \bn\|_{\partial T}\right)\nonumber\\
			\leqslant & C \beta_2 \left( \sum_{T \in \mathcal{T}_h^I} \|\nabla_w v\|_T^2 \right)^{\frac{1}{2}}\left( \sum_{T \in \mathcal{T}_h^I} \|\nabla v_0\|_T^2 \right)^{\frac{1}{2}}\nonumber\\
			&+ C  \left( \sum_{T \in \mathcal{T}_h^I} h_T^{-1} \|Q_b v_0 -v_b\|_{\partial T}^2 \right)^{\frac{1}{2}}\left( \sum_{T \in \mathcal{T}_h^I} h_T \|\beta \nabla v_0 \cdot \bn\|_{\partial T}^2 \right)^{\frac{1}{2}}\\
			\leqslant& C \beta_2\left( \sum_{T \in \mathcal{T}_h^I} \|\nabla_w v\|_T^2 \right)^{\frac{1}{2}}\left( \sum_{T \in \mathcal{T}_h^I} \|\nabla v_0\|_T^2 \right)^{\frac{1}{2}}\nonumber\\
			&+ C \beta_2 \left( \sum_{T \in \mathcal{T}_h^I} h_T^{-1} \|Q_b v_0 -v_b\|_{\partial T}^2 \right)^{\frac{1}{2}}\left( \sum_{T \in \mathcal{T}_h^I} \|\nabla v_0 \|_{T}^2 \right)^{\frac{1}{2}}\nonumber\\
			\leqslant &  C \beta_2\left( \sum_{T \in \mathcal{T}_h^I} \|\nabla v_0 \|_{T}^2 \right)^{\frac{1}{2}} \|v_h\|_{1,h}.\nonumber
		\end{align}
		Thus, the estimate \eqref{gradient_energy} holds true. The proof of the lemma is completed.
	\end{proof}
	
	\begin{lemma}\label{H1estimate}
		For {$u \in PH^{k+1}(\Omega)$} that satisfies the interface conditions \eqref{Interface_condition_1}-\eqref{Laplacian_extend_interface_conditions}, we have the following estimates
		\begin{align}
			|\ell_1(u,v)| &\leqslant C \beta_2 h^k (\|u_1\|_{k+1,\Omega_1}+\|u_2\|_{k+1,\Omega_2}) \| v \|_{1,h},\label{H1estimate1}\\
			|\ell_2(u,v)| &\leqslant C \beta_2 h^k (\|u_1\|_{k+1,\Omega_1}+\|u_2\|_{k+1,\Omega_2}) \| v \|_{1,h},\label{H1estimate2}\\	
			|\widetilde{\ell}_3(u,v)| &\leqslant C \beta_2 h^k (\|u_1\|_{k+1,\Omega_1}+\|u_2\|_{k+1,\Omega_2}) \| v \|_{1,h},\label{H1estimate3}\\
			|\ell_3(u,v)| &\leqslant C \beta_2 h^k (\|u_1\|_{k+1,\Omega_1}+\|u_2\|_{k+1,\Omega_2}) \| v \|_{1,h},\label{H1estimate8}\\
			|\ell_4(u,v)| &\leqslant C \frac{\beta_2^2}{\beta_1} h^k (\|u_1\|_{k+1,\Omega_1}+\|u_2\|_{k+1,\Omega_2}) \| v \|_{1,h},\label{H1estimate4}\\
			|\ell_5(u,v)| &\leqslant C \beta_2 h^k (\|u_1\|_{k+1,\Omega_1}+\|u_2\|_{k+1,\Omega_2}) \| v \|_{1,h},\label{H1estimate6}\\
			|\ell_6(u,v)| &\leqslant C \beta_2 h^k (\|u_1\|_{k+1,\Omega_1}+\|u_2\|_{k+1,\Omega_2}) \| v \|_{1,h},\label{H1estimate7}\\
			|s(Q_hu,v)| &\leqslant C \beta_2 h^k (\|u_1\|_{k+1,\Omega_1}+\|u_2\|_{k+1,\Omega_2}) \| v \|_{1,h}\label{H1estimate5},\\
			|s(\widetilde{Q}_hu,v)| &\leqslant C \beta_2 h^k (\|u_1\|_{k+1,\Omega_1}+\|u_2\|_{k+1,\Omega_2}) \| v \|_{1,h}\label{H1estimate9}.
		\end{align}
	\end{lemma}
	\begin{proof}
		For $\ell_1(u,v)$, according to the Cauchy-Schwarz inequality, the trace inequality and the {projection inequality \eqref{projectorestimate2}}, we get
		\begin{align}\label{H1estimate_proof_1}
			|\ell_1(u,v)|
			=&\left|\sum_{T \in \mathcal{T}_h^I} \langle Q_bv_0 - v_b, \beta \mathcal{Q}_h (\nabla u) \cdot \bn - \beta \nabla u \cdot \bn \rangle_{\partial T}\right|\nonumber\\
			\leqslant& C  \left( \sum_{T \in \mathcal{T}_h^I} h_T^{-1}\|Q_bv_0 - v_b \|_{\partial T}^2 \right)^{\frac{1}{2}} \left( \sum_{T \in \mathcal{T}_h^I} h_T \|  \beta \mathcal{Q}_h (\nabla u) \cdot \bn - \beta \nabla u \cdot \bn\|_{\partial T}^2 \right)^{\frac{1}{2}}\\
			\leqslant& C \beta_2 \| v \|_{1,h}
			\left( \sum_{T \in \mathcal{T}_h^I} \|\mathcal{Q}_h \nabla u -\nabla u\|_T^2 + h_T^2 \| \nabla(\mathcal{Q}_h \nabla u -\nabla u)\|_T^2 \right)^{\frac{1}{2}}\nonumber\\
			\leqslant& C\beta_2 h^k (\|u_1\|_{k+1,\Omega_1}+\|u_2\|_{k+1,\Omega_2}) \| v \|_{1,h}.\nonumber
		\end{align}
		Similarly, for $\ell_2(u,v)$, it follows from the Cauchy-Schwarz inequality and the {projection inequality \eqref{projectorestimate1}} that
		\begin{align}\label{H1estimate_proof_2}
			|\ell_2(u,v)|=&\left|\sum_{T \in \mathcal{T}_h^I}(\nabla(Q_0 u)- \nabla u, \beta \nabla_w v)_T\right|\nonumber\\
			\leqslant & C\beta_2 \left(\sum_{T \in \mathcal{T}_h^I} \|\nabla (Q_0 u - u)\|_T^2 \right)^{\frac{1}{2}} \left(\sum_{T \in \mathcal{T}_h^I} \| \nabla_w v\|_T^2  \right)^{\frac{1}{2}}\\
			\leqslant & C \beta_2 h^k (\|u_1\|_{k+1,\Omega_1}+\|u_2\|_{k+1,\Omega_2}) \| v \|_{1,h}.\nonumber
		\end{align}
		For $\widetilde{\ell}_3(u,v)$, by Cauchy-Schwarz inequality, we have
		\begin{align}\label{H1estimate_proof_3_IEI}
			\begin{split}
				|\widetilde{\ell}_3(u,v)|=&\left| \sum_{T \in \mathcal{T}_h^I} \langle Q_b(Q_0 u) -Q_{\delta} u, \beta \nabla_w v \cdot \bn \rangle_{\partial T} \right|\\
				\leqslant & C \left( \sum_{T \in \mathcal{T}_h^I} \|Q_b(Q_0 u) -Q_{\delta} u\|_{\partial T}^2 \right)^{\frac{1}{2}} \left(\sum_{T \in \mathcal{T}_h^I} \|\beta \nabla_w v \cdot \bn \|_{\partial T}^2 \right)^{\frac{1}{2}}.
			\end{split}
		\end{align}
		For the non-interface edge $e \in \mathcal{E}_h^n$, it follows from the definition of the $L^2$ projection operator $Q_b$, the trace inequality and the projection inequality \eqref{projectorestimate1} that 
		\begin{eqnarray}\label{H1error_varepsilon_h_proof_7_IEI}
			\begin{split}
				&\sum_{e \in \mathcal{E}_h^n} \|Q_b(Q_0 u) -Q_{\delta} u\|_{e}^2\\
				= &\sum_{e \in \mathcal{E}_h^n} \|Q_b(Q_0 u) -Q_b^e(\Pi_h u)\|_{e}^2\\
				\leqslant & \sum_{e \in \mathcal{E}_h^n} \|Q_0 u -\Pi_h u\|_{e}^2\\
				\leqslant & 2  \sum_{e \in \mathcal{E}_h^n} \|Q_0 u - u\|_{e}^2+2\sum_{e \in \mathcal{E}_h^n} \|\Pi_h u - u\|_{e}^2 \\
				\leqslant & C\sum_{T \in \mathcal{T}_h^I} \left( h_T^{-1} \|Q_0 u - u\|_{T}^2 + h_T \|\nabla(Q_0 u - u) \|_T^2\right)\\
				&+C\sum_{T \in \mathcal{T}_h^n} \left( h_T^{-1} \|\Pi_h u - u\|_{T}^2 + h_T \|\nabla(\Pi_h u - u) \|_T^2\right)\\
				\leqslant& C h^{2k+1}(\|u_1\|_{k+1,\Omega_1}^2+\|u_2\|_{k+1,\Omega_2}^2).
			\end{split}
		\end{eqnarray}
		For  $e \in \mathcal{E}_h^I$, define $e=e_1 \cup e_2$. Similarly, we get
		\begin{align}\label{H1error_varepsilon_h_proof_8_IEI}
			&\sum_{e \in \mathcal{E}_h^I} \|Q_b(Q_0 u) -Q_{\delta} u\|_{e}^2\nonumber
			=\sum_{e \in \mathcal{E}_h^I} \|Q_b(Q_0 u-u) \|_{e}^2\nonumber\\
			\leqslant & \sum_{e \in \mathcal{E}_h^I} \|Q_0 u - u\|_{\partial T}^2\\
			\leqslant & \sum_{T \in \mathcal{T}_h^I} \Big( h_T^{-1} \|Q_0 u - u\|_{T}^2 + h_T \|\nabla(Q_0 u - u) \|_{T}^2 \Big)\nonumber\\
			\leqslant& C h^{2k+1}(\|u_1\|_{k+1,\Omega_1}^2+\|u_2\|_{k+1,\Omega_2}^2).\nonumber
		\end{align}	
		Substituting the estimates \eqref{H1error_varepsilon_h_proof_7_IEI}-\eqref{H1error_varepsilon_h_proof_8_IEI} into  \eqref{H1estimate_proof_3_IEI} and using Lemma \ref{inverse_interface_element} leads to
		\begin{align}\label{H1error_varepsilon_h_proof_9_IEI}
			&|\widetilde{\ell}_3(u,v)|\nonumber\\
			\leqslant& C\beta_2 h^{k+\frac{1}{2}}(\|u_1\|_{k+1,\Omega_1}+\|u_2\|_{k+1,\Omega_2})\left(\sum_{T \in \mathcal{T}_h^I} h_T^{-1} \|\nabla_w v\|_T^2 \right)^{\frac{1}{2}}\\
			\leqslant& C\beta_2 h^k (\|u_1\|_{k+1,\Omega_1}+\|u_2\|_{k+1,\Omega_2}) \| v \|_{1,h}.\nonumber
		\end{align}	
		Similarly, the following estimate holds true
		\begin{align*}
			|\ell_3(u,v)|\leqslant C\beta_2 h^k (\|u_1\|_{k+1,\Omega_1}+\|u_2\|_{k+1,\Omega_2}) \| v \|_{1,h}.
		\end{align*}	
		For $\ell_4(u,v)$, by the fact that $\sum_{T \in \mathcal{T}_h^I} \langle Q_b(\beta\nabla u \cdot \bn), v_0 -Q_b v_0 \rangle_{\partial T}=0$ and the Cauchy-Schwarz inequality, we obtain
		\begin{eqnarray}\label{H1estimate_proof_4}
			\begin{split}
				|\ell_4(u,v)|=&\left| \sum_{T \in \mathcal{T}_h^I} \langle \beta\nabla u \cdot \bn, v_0 -Q_b v_0 \rangle_{\partial T} \right|\\
				=&\left| \sum_{T \in \mathcal{T}_h^I} \langle \beta\nabla u \cdot \bn - Q_b(\beta\nabla u \cdot \bn), v_0 -Q_b v_0 \rangle_{\partial T} \right|\\
				\leqslant&C \left( \sum_{T \in \mathcal{T}_h^I} \|\beta\nabla u \cdot \bn - Q_b(\beta\nabla u \cdot \bn)\|_{\partial T}^2 \right)^{\frac{1}{2}} 
				\left( \sum_{T \in \mathcal{T}_h^I} \|v_0 -Q_b v_0 \|_{\partial T}^2 \right)^{\frac{1}{2}}.
			\end{split}
		\end{eqnarray}
		For $e \in \mathcal{E}_h^n$, according to the trace inequality and the {projection inequality \eqref{projectorestimate1}}, we get
		\begin{eqnarray}\label{H1estimate_proof_5}
			\begin{split}
				&\|\beta\nabla u \cdot \bn - Q_b^e(\beta\nabla u \cdot \bn)\|_e^2\\
				\leqslant& C 	\|\beta\nabla u \cdot \bn - Q_0^{k-1}(\beta\nabla u \cdot \bn)\|_e^2\\
				\leqslant& C h_T^{-1}\|\beta\nabla u \cdot \bn - Q_0^{k-1}(\beta\nabla u \cdot \bn)\|_{T}^2+ C h_T\|\nabla(\beta\nabla u \cdot \bn - Q_0^{k-1}(\beta\nabla u \cdot \bn))\|_{T}^2\\
				\leqslant & C\beta_2^2 h_T^{2k-1}(\|u_1\|_{k+1,\Omega_1}^2+\|u_2\|_{k+1,\Omega_2}^2).
			\end{split}
		\end{eqnarray}
		Similarly, for the interface edge $e \in \mathcal{E}_h^I$, we have
		\begin{eqnarray}\label{H1estimate_proof_6}
			\begin{split}
				\|\beta\nabla u \cdot \bn - Q_b(\beta\nabla u \cdot \bn)\|_e^2
				\leqslant C\beta_2^2 h_T^{2k-1} (\|u_1\|_{k+1,\Omega_1}^2+\|u_2\|_{k+1,\Omega_2}^2).
			\end{split}
		\end{eqnarray}
		By the estimate \eqref{gradient_energy}, we obtain
		\begin{eqnarray}\label{H1estimate_proof_7}
			\begin{split}
				\sum_{T \in \mathcal{T}_h^I} \|v_0 -Q_b v_0 \|_{\partial T}^2
				\leqslant Ch\sum_{T \in \mathcal{T}_h^I}\|\nabla v_0\|_T^2\leqslant Ch\frac{\beta_2^2}{\beta_1^2}\|v\|_{1,h}^2.
			\end{split}
		\end{eqnarray}
		Substituting the above two estimates \eqref{H1estimate_proof_5}-\eqref{H1estimate_proof_7} into the estimate \eqref{H1estimate_proof_4}, the estimate \eqref{H1estimate4} holds true.
		
		For $\ell_5(u,v)$, it follows from the fact that $\sum_{e \in \mathcal{E}_h^n} \langle Q_b(\beta\nabla u \cdot \bn_e) , v- Q_b v \rangle_e=0$, the Cauchy-Schwarz inequality, the trace inequality and the projection inequality \eqref{projectorestimate1} that 
		\begin{align}\label{H1estimate_proof_9}
			|\ell_5(u,v)|=&\left| \sum_{e \in \mathcal{E}_h^n} \langle \beta\nabla u \cdot \bn_e, v- Q_b v\rangle_e \right|\nonumber\\
			=&\left|\sum_{e \in \mathcal{E}_h^n} \langle \beta\nabla u \cdot \bn_e - Q_b(\beta\nabla u \cdot \bn_e), v- Q_b v\rangle_e\right|\nonumber\\
			\leqslant &C\left(\sum_{e \in \mathcal{E}_h^n} \|\beta\nabla u \cdot \bn_e - Q_b(\beta\nabla u \cdot \bn_e) \|_e^2 \right)^{\frac{1}{2}}
			\left(\sum_{e \in \mathcal{E}_h^n} \|v- Q_b v \|_e^2 \right)^{\frac{1}{2}}\\
			\leqslant &C\left(\sum_{e \in \mathcal{E}_h^n} \|\beta\nabla u \cdot \bn_e - Q_0^{k-1}(\beta\nabla u \cdot \bn_e) \|_e^2 \right)^{\frac{1}{2}}
			\left(\sum_{e \in \mathcal{E}_h^n} \|v- Q_b v \|_e^2 \right)^{\frac{1}{2}}\nonumber\\
			\leqslant & C \beta_2 h^{k}(\|u_1\|_{k+1,\Omega_1}+\|u_2\|_{k+1,\Omega_2}) \| v \|_{1,h}.\nonumber
		\end{align}
		For $\ell_6(u,v)$, using the Cauchy-Schwarz inequality, the projection inequality \eqref{projectorestimate1} and the interpolation estimate, we get
		\begin{align}
			\begin{split}
				|\ell_6(u,v)|=&\left| \sum_{T \in \mathcal{T}_h^n}(\beta \nabla(\Pi_h u-u),\nabla v)_T  \right|\\
				\leqslant &  C\beta_2 \sum_{T \in \mathcal{T}_h^n} \|\nabla(\Pi_h u-u)\|_T\|\nabla v\|_T\\
				\leqslant& C\beta_2  \left( \sum_{T \in \mathcal{T}_h^n} \|\nabla(\Pi_h u-u)\|_T^2 \right)^{\frac{1}{2}}\left(\sum_{T \in \mathcal{T}_h^n} \|\nabla v\|_T^2 \right)^{\frac{1}{2}}  \\
				\leqslant & C\beta_2 h^k (\|u_1\|_{k+1,\Omega_1}+\|u_2\|_{k+1,\Omega_2}) \| v \|_{1,h}.
			\end{split}
		\end{align}	
		For $s(Q_hu,v)$, according to the Cauchy-schwarz inequality, the definition of the $L^2$ projection operator $Q_b$, the trace inequality, and the projection inequality \eqref{projectorestimate1}, we have
		\begin{eqnarray}\label{H1estimate_proof_10}
			\begin{split}
				|s(Q_hu,v)|=& \left|\sum_{T \in \mathcal{T}_h} \langle \beta (Q_b(Q_0 u)-Q_bu), Q_b v_0 -v_b  \rangle_{\partial T} \right|\\
				\leqslant & C \beta_2 \left(\sum_{T \in \mathcal{T}_h} h_T \|Q_0 u -u \|_{\partial T}^2 \right)^{\frac{1}{2}}
				\left(\sum_{T \in \mathcal{T}_h} h_T^{-1}\| Q_b v_0 -v_b \|_{\partial T}^2 \right)^{\frac{1}{2}}\\
				\leqslant& C \beta_2 \left( \sum_{T \in \mathcal{T}_h} \|Q_0 u -u\|_T^2+ h_T^2\|\nabla (Q_0 u -u)\|_T^2
				\right) ^{\frac{1}{2}}\left(\sum_{T \in \mathcal{T}_h} h_T^{-1}\| Q_b v_0 -v_b \|_{\partial T}^2 \right)^{\frac{1}{2}}\\
				\leqslant & C \beta_2 h^{k}(\|u_1\|_{k+1,\Omega_1}+\|u_2\|_{k+1,\Omega_2})\| v \|_{1,h}.
			\end{split}
		\end{eqnarray} 
		Similarly, the estimate of $s(\widetilde{Q}_hu,v)$ can be obtained. 
		The proof of the lemma is completed.
	\end{proof}	
	


\end{document}